\documentclass{article}
\usepackage[a4paper, margin=1in]{geometry} 
\usepackage{graphicx} 
\usepackage{forest,tikz,amsfonts,amsmath,amssymb,amsthm}

\usepackage[a4paper, margin=1in]{geometry} 

\usepackage{subcaption} 

\usetikzlibrary{shapes.geometric, arrows}

\usepackage{ytableau}
\usepackage{multicol}

\usepackage{enumitem}

\usepackage{dsfont}

\usepackage{soul}

\usepackage{hyperref}
\hypersetup{
    colorlinks=true,
    linkcolor=blue,
    filecolor=magenta,      
    urlcolor=cyan,
    pdftitle={Overleaf Example},
    pdfpagemode=FullScreen,
    }

\usepackage{cleveref}

\def\V#1{{\mathbf #1}}

\def\Uniform{\operatorname{Uniform}}
\def\Haar{\operatorname{Haar}}

\def\B{\operatorname{B}}
\def\H{\operatorname{H}}
\def\O{\operatorname{O}}
\def\U{\operatorname{U}}
\def\SO{\operatorname{SO}}

\def\sgn{\operatorname{sgn}}

\DeclareMathOperator*{\argmax}{arg{,}max}

\def\t{\theta}
\def\vep{\varepsilon}
\def\trace{\operatorname{tr}}

\newtheorem{lemma}{Lemma}
\newtheorem{proposition}{Proposition}
\newtheorem{theorem}{Theorem}
\newtheorem{corollary}{Corollary}

\newtheorem{remark}{Remark}

\def\red#1{#1}

\allowdisplaybreaks

\title{Complete pivoting growth of butterfly matrices and butterfly Hadamard matrices}
\author{John Peca-Medlin\thanks{Department of Mathematics, University of California, San Diego, \href{mailto:jpecamedlin@ucsd.edu}{jpecamedlin@ucsd.edu}}}
\date{}

\begin{document}

\maketitle

\begin{abstract}
The growth problem in Gaussian elimination (GE) remains a foundational question in numerical analysis and numerical linear algebra. Wilkinson resolved the growth problem in GE with partial pivoting (GEPP) in his initial analysis from the 1960s, while he was only able to establish an upper bound for the GE with complete pivoting (GECP) growth problem. The GECP growth problem has seen a spike in recent interest, culminating in improved lower and upper bounds established by Bisain, Edelman, and Urschel in 2023, but still remains far from being fully resolved. Due to the complex dynamics governing the location of GECP pivots, analysis of GECP growth for particular input matrices often estimates the actual growth rather than computes the growth exactly. We present a class of dense random butterfly matrices \red{for which we can compute} the exact GECP growth. We extend previous results that established exact growth computations for butterfly matrices when using GEPP and GE with rook pivoting (GERP) to now also include GECP \red{for structured subclasses of inputs}. Moreover, we present a new method to construct random Hadamard matrices using butterfly matrices. 

\end{abstract}

\section{Introduction}

Gaussian elimination (GE) remains one of the most used approaches to solve linear systems in modern applications. For instance, GE with partial pivoting (GEPP) is the default solver in MATLAB when using the backslash operator with a general \red{dense} input matrix. Additionally, GE is a staple of introductory linear algebra courses (although, based on anecdotal evidence, it may not be a favorite among all students).

GE iteratively uses elimination updates below the diagonal on an initial input matrix $A \in \mathbb R^{n \times n}$ to transform $A$ into an upper triangular linear system, which eventually builds the matrix factorization $A = LU$, where $L$ is a unit lower triangular matrix and $U$ is upper triangular. Each GE step transforms $A = A^{(1)}$ into the upper triangular form $U = A^{(n)}$ using $n-1$ total GE steps, where $A^{(k)}$ denotes the intermediate form of $A$ that has all zeros below the first $k-1$ entries. The lower triangular matrix $L$ is built up by successively using the \textit{pivot}, i.e., the \red{leading diagonal} entry in the untriangularized system, to scale the leading row to zero out all remaining entries below the pivot. This can then be used to solve $A \V x = \V b$ by solving the computationally simpler triangular linear systems $L \V y = \V b$ and $\V U \V x = \V y$ using forward and backward substitutions. General nonsingular $A$ will not have a $LU$ factorization if any leading minors are singular, i.e., if $\det [(A_{ij})_{i,j=1}^k] = 0$ for some $k \le n-1$. Moreover, numerical stability of GE when using floating-point arithmetic\footnote{We will use the IEEE standard model for floating-point arithmetic.} is sensitive to any elimination steps that involve division by numbers close to 0. Hence, combining GE with certain pivoting schemes is preferable even if not necessary. 

GEPP remains the most popular pivoting scheme, which combines a pivot search in the leading column of the untriangularized system to then also use a row transposition so that the pivot is maximal in magnitude. This results in the GEPP factorization $PA = LU$ where $P$ is a permutation matrix, with corresponding permutation 
\begin{equation}\label{eq: perm form}
    \sigma(A) = (n \ i_n) \cdots (2 \ i_2)(1 \ i_1),
\end{equation}
where $i_k \ge k$ is defined by $i_k = \min\{\argmax_{j \ge k} |A^{(k)}_{jk}|\}$ (corresponding to the closest maximal in magnitude entry to the original pivot entry). Moreover, also $\|L\|_{\max} = 1$ (using $\|C\|_{\max} = \max_{i,j} |C_{ij}|$ for the max-norm on matrices) since then $L_{ij} = A^{(j)}_{ij}/A^{(j)}_{jj}$ satisfies $|L_{ij}| \le 1$ for all $i > j$.  We will focus in this paper on GE with complete pivoting (GECP), that uses both a corresponding row and column pivot to ensure the  pivot is maximal in the entire untriangularized system (rather than just the first column). GECP factorizations are then of the form $PAQ=LU$ where both $P,Q$ are permutation matrices, for permutations similarly in form  \eqref{eq: perm form}. Moreover, we will assume GECP uses the default \textit{lexicographic} tie-breaking pivot search strategy assuming a \textit{column-major} traversal, i.e., at GECP step $k$, use the corresponding row and column swaps for the location of the first maximal magnitude entry in the minimal column and then minimal row entry. This tie-breaking strategy aligns with MATLAB's column-major storage order, which stacks matrix columns from left to right. It then performs a GEPP pivot search on the extended column to determine the maximal element.

Using floating-point arithmetic, the numerical stability of computed solutions using GE on well-conditioned matrices relies heavily on the growth factor,
\begin{equation}\nonumber
    \rho(A) = \frac{\max_k \|A^{(k)}\|_{\max}}{\|A\|_{\max}}.
\end{equation}
(The growth factor uses intermediate \textit{computed} matrix factors, $\overline{A^{(k)}}$ with finite precision; we will refrain from highlighting the distinction between computed and exact intermediate factors throughout, noting also that computed and actual matrix factors align in exact arithmetic.) Using GECP, then $\max_k \|A^{(k)}\|_{\max} = \max_j |U_{jj}|$ so that $\rho(A) = \max_j |U_{jj}|/|U_{11}|$. For a quick demonstration of  how the growth factor can be used in error analysis of computed solutions, we consider a computed solution $\hat{\V x}$ in floating-point arithmetic using the GE factorization $PAQ = LU$ to the linear system $A \V x = \V b$ where $A \in \mathbb R^{n\times n}$ and $\V b \in \mathbb R^n$. Then, as seen in \cite{Hi02}, the relative error then satisfies
\begin{equation}\label{eq: gf ineq}
    \frac{\|\V x - \hat{\V x}\|_\infty}{\|\V x\|_\infty} \le 4 n^2 \epsilon_{\operatorname{machine}} \kappa_\infty(A) \rho(A)
\end{equation}
where $\kappa(A) = \|A\|_\infty \|A^{-1}\|_\infty$ is the $L^\infty$-induced condition number of $A$ where $\epsilon_{\operatorname{machine}}$ is the machine-epsilon that denotes the minimal positive number such that the floating-point representation  $\operatorname{fl}(1 + \epsilon_{\operatorname{machine}}) \ne 1$ (e.g., $\epsilon_{\operatorname{machine}} = 2^{-52}$ using double precision). Hence, on well-conditioned linear systems, error analysis of computed solutions can focus on analysis of the corresponding growth factors.

The growth problem is the moniker given to the ordeal of determining the largest possible growth factors on matrices of a given order for a given pivoting scheme. In \cite{Wi61,Wi65}, Wilkinson resolved in short order the growth problem using GEPP, by presenting a $n\times n$ matrix  such that the GEPP growth factor attains the worst-case possible growth of $2^{n-1}$. In the same work, Wilkinson only established a \red{quasipolynomial general} upper bound for the GECP growth factor of order $n^{0.25 \log n(1 + o(1))}$. This bound, however, is believed to be far from optimal. The GECP growth problem (which will be referred to from here on simply as ``the growth problem'') is only resolved for $n = 1,2,3,4$, with maximal growth, respectively, of 1, 2, 2.25, 4, while for $n = 5$ we only have the upper bound $4 \frac{17}{18}\approx 4.94$ \red{that has recently been lowered to 4.84 in \cite{chen2026largest} but is conjectured to be about 4.1325} (see \cite{EdUr23} for a historical overview of the growth problem). Wilkinson's \red{general} upper bound remained essentially the best in the business for over 6 decades until 2023, when Bisain, Edelman, and Urschel used an improved Hadamard inequality argument to lower this now to about $n^{0.2079 \log n (1+o(1))}$ \cite{Bisain}. When compared with the current best linear lower bounds, also established by Edelman and Urschel earlier in 2023 in \cite{EdUr23}, the growth problem still is far from being resolved.

The difficulty in the growth problem largely comes down to the increasingly complex dynamics between entries for each iterative pivot search and subsequent elimination step. Instead of considering the growth problem across all nonsingular matrices, one can then focus on the growth problem for certain structured matrices. Hadamard matrices have their own rich history with the growth problem. $H_n \in \{\pm 1\}^{n \times n}$ is a Hadamard matrix if $H_n H_n^\top = n \V I_n$, where necessarily $n = 1,2$ or $n$ is a multiple of 4 (see \cite{DaPe88} for a larger historical overview). Hadamard matrices have a wide list of applications, including optimal design theory, coding theory, and graph theory (see, for example, \cite{Hi02,Kravvaritis_2016}, and references therein for more background). A famous open problem with Hadamard matrices is the existence of such matrices for all multiples of 4, with 768 currently being the smallest such order that no known Hadamard matrices have been found yet (\cite{Dok08} resolved the previously lower bound of 764 in 2008). A similarly famous open problem for numerical analysis involves the growth problem for Hadamard matrices, where it is believed $\rho(H_n) = n$. This has been established for all Sylvester Hadamard matrices (see \cite{DaPe88}), while for general Hadamard matrices this has only been proven up to $n = 16$ \cite{KrMi09}. The growth problem for Hadamard matrices is a sub-question for the growth problem when restricted to orthogonal matrices. For instance, the orthogonal growth problem remains open in GEPP even, which was recently explored in \cite{p24_orth}.

We are interested in studying the Hadamard growth problem, as viewed in a continuous analogue version using butterfly matrices. (See \Cref{sec:rbm} for a definition of butterfly matrices.) Butterfly matrices were introduced by D. Stott Parker as a means to remove the need of pivoting altogether using GE. For example, to solve $A \V x = \V b$, one can use independent and identically distributed (iid) random butterfly matrices $U,V$ to instead solve the equivalent linear systems $UAV^* \V y = U\V b$ and $\V x = V^*\V y$, but now using GE with no pivoting (GENP) \cite{Pa95}. The primary utility in using butterfly matrices to solve this preceding transformed linear system is their recursive structure allows matrix-vector products to be performed in $\mathcal O(n \log n)$ FLOPs;  so these transformations can be performed without impacting the $\mathcal O(n^3)$ leading order complexity of GE itself. 

In \cite{PT23}, the authors studied the stability properties of these butterfly transformed systems for particular initial linear systems by studying their corresponding computed growth factors when using GENP, GEPP, GE with rook pivoting (GERP, that uses iterative column and then row searches for a pivot candidate that maximizes \textit{both} its column and row) and GECP. Moreover, they provide a full distributional description of the GENP, GEPP, and GERP growth factors for a particular subclass of random butterfly matrices themselves. One of our main contributions of this paper is to now extend this full distributional description to now also include GECP growth factors for a further subclass of these butterfly matrices. We also explore structural properties of GECP factorizations using different tie-breaking pivot search strategies. Additionally, we re-contextualize the discussion of butterfly matrices in terms of their growth maximizing properties for a further subclass of scaled Hadamard matrices, which we will refer to as butterfly Hadamard matrices. We then introduce a direct means to sample random butterfly Hadamard matrices.

\subsection{Outline of results}

\red{The starting point for this work is \Cref{thm: gf gepp} from \cite{PT23}, which provides a full distributional description of the GEPP growth factor for simple butterfly matrices. This paper extends that result by identifying a class of butterfly matrices that do not require any row or column pivots under GECP, i.e., they are completely pivoted. Our main theoretical contribution is \Cref{thm: new}, which establishes this property and, in addition, provides the distribution of the associated random growth factors for matrices that preserve a monotonicity property of the input angle vector.} 

\Cref{sec:rbm} is organized to first give sufficient background and highlight particular properties of random butterfly matrices, which is then utilized in the main technical approach given in \Cref{prop: gecp} to maximize growth at intermediate GECP steps \red{needed for the proof of \Cref{thm: new}}. \Cref{sec: experiments} provides additional supporting experiments that consider GECP factorizations for butterfly matrices with different initial input orders, as well as with and without using GECP with an additional tolerance parameter for determining intermediate pivot candidates when using floating-point arithmetic. \Cref{sec: had} focuses on a subclass of scaled Hadamard matrices within the butterfly matrices, called the butterfly Hadamard matrices. This class represents a rich subclass of Hadamard matrices (of order $2^{\mathcal O(N \log N)}$), that necessarily maximizes the growth factor in \Cref{thm: new} for a further subclass of simple scalar butterfly matrices. Moreover, we present a simple method to construct a particular butterfly Hadamard matrix for a given input random butterfly matrix \red{using the componentwise sign function}.

\section{Growth factors of random butterfly matrices}
\label{sec:rbm}

Butterfly matrices are a family of recursive orthogonal transformations that were introduced by D. Stott Parker in 1995 as a means of accelerating common computations in linear algebra \cite{Pa95}. We will use a particular family of butterfly matrices built up using rotation matrices, studied in more detail previously in \cite{PT23,Tr19}. We will now review particular definitions of these butterfly matrices, along \red{with} certain GE matrix factorization forms and growth factor results for particular butterfly matrices.

\subsection{Random butterfly matrices}

An order $N=2^n$ \textit{butterfly matrix} $B$ takes the recursive form
\begin{equation}
\label{eq:bm_def}
    B = \begin{bmatrix} C & S\\-S & C\end{bmatrix} \begin{bmatrix} A_1 & \V0 \\\V 0&A_2\end{bmatrix} = \begin{bmatrix} C A_1 & S A_2\\-S A_1 & C A_2\end{bmatrix},
\end{equation}
using symmetric and commuting matrices, $C,S$,  that satisfy the matrix Pythagorean identity $C^2+S^2=\V I_{N/2}$ and $A_1,A_2$ are order $N/2$ butterfly matrices, where we define $B = 1$ if $N = 1$.  If $A_1=A_2$ at each recursive step, then the resulting butterfly matrices are called \textit{simple} butterfly matrices. By construction, necessarily $B \in \SO(N)$ and also $(C,S)$ have corresponding eigenvalue pairs $(\cos\theta,\sin\theta)$ for some angle $\theta$. The primary utility for butterfly matrices is due to the fast matrix-vector multiplication afforded by this recursive structure, which uses $\mathcal O(Nn)$ FLOPs assuming $C,S$ have $\mathcal O(N)$ matrix-vector multiplication complexity; this can be ensured by then using diagonal $C,S$ at each recursive step. 

Let $\B(N)$ and $\B_s(N)$ denote the order $N$ \textit{scalar} nonsimple and  simple butterfly matrices, formed using scalar  $(C_k,S_k)=(\cos\theta_k,\sin\theta_k) \V I_{2^{k-1}}$ at each recursive step $k$ for some angle $\theta_k$. Similarly, let $\B^{(D)}(N)$ and $\B_s^{(D)}(N)$ denote the order $N$ \textit{diagonal} nonsimple and simple butterfly matrices, formed using diagonal $(C_k,S_k) = \bigoplus_{j=1}^{2^{k-1}}(\cos \theta_k^{(j)},\sin \theta_k^{(j)})$ for each recursive step $k$.

Recall the Kronecker product of $A \in \mathbb R^{n_1 \times m_1}, B \in \mathbb R^{n_2 \times m_2}$ is the matrix $A \otimes B \in \mathbb R^{n_1n_2 \times m_1m_2}$ of the form
\begin{equation}\nonumber
    A \otimes B = \begin{bmatrix}
        A_{11}B &  \cdots & A_{1m_1} B\\
        \vdots  & \ddots & \vdots\\
        A_{n_1 1} B & \cdots & A_{n_1m_1} B
    \end{bmatrix}.
\end{equation}
We will also use the notation $A \oplus B \in \mathbb R^{(n_1 + n_2)\times(m_1 + m_2)}$ to denote the block diagonal matrix with diagonal blocks $A,B$. Using these notations, $B_1 \in \B(N)$ and $B_2 \in \B_s(N)$ have \eqref{eq:bm_def} take the forms
\begin{equation}\nonumber
    B_1 = (R_\theta \otimes \V I_{N/2})(A_1 \oplus A_2), \qquad B_2 =\bigotimes_{j=1}^n R_{\theta_j}, \qquad \mbox{with} \qquad R_\theta = \begin{bmatrix}
        \cos\theta & \sin \theta \\ - \sin\theta & \cos \theta
    \end{bmatrix}
\end{equation}
denoting the standard (clockwise) rotation matrices. Further recall the Kronecker product satisfies the \textit{mixed-product property}, where
\begin{equation}\nonumber
    (A \otimes B)(C \otimes D) = (AC) \otimes (BD)
\end{equation}
when matrix dimensions are compatible. Moreover, both $\otimes$ and $\oplus$ preserve certain matrix structures and operators, such as permutation matrices, triangular forms, inverses, and transposes. Additionally, there exist \textit{perfect shuffle} permutations (see \cite{perfectShuffle}) $P,Q$ such that $P(A \otimes B)Q = B \otimes A$; if $A,B$ are both square, then $Q = P^\top$. Hence, for $Q_n$ the perfect shuffle matrices such that $Q_n(A \otimes \V I_{N/2})Q_n^\top = \V I_{N/2} \otimes A$ for $A \in \mathbb R^{2\times 2}$, then the diagonal butterfly matrices $B_1 \in \B^{(D)}(N)$ and $B_2 \in \B^{(D)}_s(N)$ have \eqref{eq:bm_def} take the form
\begin{equation}\label{eq: diag form}
    B_1 = Q_n\left(\bigoplus_{j=1}^{N/2} R_{\theta_j}\right)Q_n^\top \cdot (A_1 \oplus A_2)
\end{equation}
for $A_1,A_2 \in \B^{(D)}(N/2)$, where additionally $A_1 \oplus A_2 = \V I_2 \otimes A$ in the simple case  $A_1=A_2 = A \in \B^{(D)}_s(N/2)$.  A straightforward check confirms the number $M$ of input angles needed to generate a butterfly matrix from $\B_s(N),\B(N),\B_s^{(D)}(N),\B^{(D)}(N)$ is, respectively, $n$, $N-1$, $N-1$, and $nN/2$. We will write $B(\boldsymbol{\t})$ with $\boldsymbol{\t} \in [0,2\pi)^M$ then for each respective type of butterfly matrix.

Random butterfly matrices are formed using random input angles. Let $\Sigma$ be a collection of generating pairs $(C_k,S_k)$ satisfying the above butterfly properties with random input angles. Then $\B(N,\Sigma)$ and $\B_s(N,\Sigma)$ will denote the ensembles of random butterfly matrices and random simple butterfly matrices formed by constructing a butterfly matrix by independently sampling $(C,S)$ from $\Sigma$ at each recursive step. Let
\begin{align*}
    \Sigma_S &= \{(\cos\theta^{(k)},\sin\theta^{(k)})\V I_{2^{k-1}}: \theta^{(k)} \mbox{ iid }  \operatorname{Uniform}([0,2\pi)), k\ge 1\}
\quad \mbox{and}\\
    \Sigma_D &= \{ \bigoplus_{j=1}^{2^{k-1}}(\cos\t_j^{(k)},\sin\t_j^{(k)}): \theta^{(k)}_j \mbox{ iid }  \operatorname{Uniform}([0,2\pi)),  k\ge 1\}.
\end{align*}
Note from the mixed-product property, $\B_s(N) = \bigotimes^n \SO(2)$ is an abelian subgroup of $\SO(N)$, and hence has a Haar probability measure. Using uniform input angles then leads to induced uniform measures on the push-forward transformations, so that:
\begin{proposition}[\cite{Tr19}]
\label{prop:haar_butterfly}
    $\B_s(N,\Sigma_S) \sim \Haar(\B_s(N))$.
\end{proposition}
We will then refer to $\B_s(N,\Sigma_S)$ as the \textit{Haar-butterfly matrices}, as this explicit construction provides a method to directly sample from this distribution.

A natural question using butterfly transformations on linear systems is to determine how far away is the transformed system from the input system. Since butterfly matrices are orthogonal, then multiplication then is equivalent to taking a one-step orthogonal walk with initial location at the input matrix; so multiplying by more butterfly matrices then corresponds to a longer orthogonal walk. 

We can measure how far a step is in this orthogonal walk by just viewing how much does perturbing the input angles change the overall matrix. This can be answered explicitly by computing bounds for $\|B(\boldsymbol\t) - B(\boldsymbol \t + \boldsymbol \varepsilon)\|_F$, which show the map $\boldsymbol \t \mapsto B(\boldsymbol \t)$ is Lipschitz continuous. 

\begin{proposition}
\label{prop: lipschitz}
The maps $\boldsymbol\t \mapsto B(\boldsymbol\t)$ for $B(\boldsymbol\t)$ in $\B_s(N),\B(N),\B_s^{(D)}(N)$ and $\B^{(D)}(N)$ are each Lipschitz continuous with respect to the $2$-norm to Frobenius normed spaces, with respective Lipschitz constants of $\sqrt{N}$, $\sqrt{2(N-1)}$, $\sqrt{2(N-1)}$ and $\sqrt{2n}$.
\end{proposition}

In particular, this shows the butterfly models comprise connected manifolds of dimension $M$ for each respective type of butterfly matrices. A proof of \Cref{prop: lipschitz} is found in \Cref{sec: appendix}.

\subsection{GECP growth factor of butterfly matrices}

The growth factor is an important statistic that controls the precision when using GE to solve linear systems, as seen in \eqref{eq: gf ineq}. In particular, for well-conditioned linear systems, error analysis on GE can focus on analysis of the growth factor itself. The study of growth factors on random linear systems has its own rich history. 

The previous focus on worst-case behavior to gauge algorithm applications led earlier researchers to focus more on GECP when using GE, due to its sub-exponential worst-case upper bound, and hence the ongoing interest in the GECP growth problem. However, despite potentially exponential growth, typical GEPP growth remained much smaller in practice. Understanding why this is the case has remained a continued question of interest in numerical analysis. 

A good starting point on the study of random growth factors is the \textit{average-case analysis} of Schreiber and Trefethen in 1990 \cite{TrSc90}. Their analysis shifted the focus of the growth problem from worst-case behavior to instead look at the average behavior on different random matrix ensembles. Using different pivoting schemes, they presented empirical estimates of the first moment on iid random Gaussian ensembles of size up to $10^4$, showing GEPP typically maintained polynomial growth factors of order $n^{2/3}$ while GECP had growth of order $n^{1/2}$. Understanding the behavior of growth factors on fixed linear systems using additive Gaussian perturbations falls under the heading of \textit{smoothed analysis} (see, e.g., \cite{Spielman_Teng_2009} for a general overview of using smoothed analysis to analyze algorithms). A large step forward in understanding the overall behavior of GE was accomplished by Sankar, Speilman, and Teng who provided a full smoothed analysis of GENP in 2006 \cite{SST06}. However, the increasing complexity of pivoting prevented their methods from \red{extending to a full smoothed analysis of GEPP}. For instance, only partial smoothed analysis results of GEPP were  carried out by Sankar in his Ph.D. thesis \cite{S04}. Huang and Tikhomirov gave the most significant contribution to improving our understanding of GEPP in their improved average-case analysis using iid Gaussian matrices, where they provide high probability polynomial bounds on the growth factors \cite{HT23}. Again, a full smoothed analysis of GEPP remains out of reach still, as their methods only establish improved growth factor bounds on additive Gaussian perturbations of the all zeros matrix (compared to Sankar's previous work). Other recent work has explored dynamics between GEPP and GECP growth factors, as well as studying growth behavior using multiplicative random orthogonal perturbations \cite{p24_orth}.

While previous results using random growth factors focused on empirical moment estimates and bounds, the authors in \cite{PT23} take advantage of the structural properties of Haar-butterfly matrices to then provide full distributional descriptions for certain growth factors, including GENP, GEPP, and GERP. A similar approach was taken in \cite{EdUr23} for improved lower bounds on the GERP growth problem, where Kronecker products were further utilized as they preserve certain maximal growth properties. (However, this does not extend to GECP.)

We recall the statement of the theorem (cf. \cite[Theorem 2.3]{PT23}) that resolves the GENP, GEPP, and GERP growth problem for Haar-butterfly matrices:
\begin{theorem}[\cite{PT23}]\label{thm: gf gepp}
    Let $B \sim \B_s(N,\Sigma_S)$ and $X \sim \operatorname{Cauchy}(1)$. Then $\rho(B) \sim \prod_{j=1}^n (1 + Y_j^2)$ where $Y_j$ are iid, with $Y_j \sim |X|$ when using GENP and $Y_j \sim (|X| \mid |X| \le 1)$ when using GEPP or GERP. Moreover, using any pivoting scheme then $\rho(B) \ge \rho^{\operatorname{GEPP}}(B)$.
\end{theorem}
The proof of \Cref{thm: gf gepp} utilized $\tan \t \sim X$ and $\max(|\tan \t|,|\cos\t|) \sim (|X| \mid |X| \le 1)$ for $X \sim \operatorname{Cauchy}(1)$ when $\t \sim \Uniform([0,2\pi))$, in addition to certain maximizing pivoting properties among the intermediate GE forms of $B \sim \B_s(N,\Sigma_S)$, which can be derived from the following lemma:

\begin{lemma}[\cite{PT23}]
\label{lemma:Bk}
    Suppose $A \in \mathbb R^{N/2\times N/2}$ has an $LU$ factorization using GENP. Let
    $$
    B = \begin{bmatrix}
    \cos\theta A & \sin\theta A\\-\sin\theta A & \cos\theta A
    \end{bmatrix} = R_{\t} \otimes A
    $$
    for $\cos\theta \ne 0$. Then $B$ has an $LU$ factorization using GENP. Moreover, if $k \le N/2$, then
    \begin{equation}
    \label{eq:Bk first half}
        B^{(k)} = \begin{bmatrix}
        \cos\theta A^{(k)} & \sin\theta A^{(k)}\\
        -\sin\theta \begin{bmatrix}
        \V0&\V0\\\V0&\V I_{N/2-k+1}
        \end{bmatrix} A^{(k)}
        & \sec\theta \left(A - \sin^2\theta \begin{bmatrix}
        \V0&\V0\\ \V0&\V I_{N/2-k+1}
        \end{bmatrix}A^{(k)} \right)
        \end{bmatrix}.
    \end{equation}
    If $k=N/2+j$ for $j \ge 1$, then
    \begin{equation}
    \label{eq:Bk second half}
        B^{(k)} = \begin{bmatrix}
        \cos\theta A^{(N/2)} & \sin\theta A^{(N/2)}\\\V 0 &\sec\theta A^{(j)}
        \end{bmatrix}.
    \end{equation}
\end{lemma}
In particular, \Cref{lemma:Bk} then yields 
\begin{equation}\label{eq: gf angle vec}
    \rho(B(\boldsymbol{\t})) = \prod_{j=1}^n \max(|\sec \t_j|,|\csc \t_j|)^2 = \prod_{j=1}^n (1 + \max(|\tan \t_j|,|\cot\t_j|)^2)
\end{equation}
for $B(\boldsymbol{\t}) \in \B_s(N)$ if using GEPP (or GERP). It follows then this product is invariant under any permutation of the indices for the input angles. \red{Let $S_n$ denote the symmetric group of permutations of length $n$.}
\begin{corollary}\label{cor: gf inv sym}
    Let $B(\boldsymbol{\t}) \in \B_s(N)$ and $\sigma \in S_n$. If $\tilde \t_j = \t_{\sigma(j)}$ \red{for each $j = 1,2,\ldots,n$}, then $\rho(B(\boldsymbol{\t})) = \rho(B(\tilde{\boldsymbol{\t}}))$.
\end{corollary}

\begin{figure}[t] 
\centering
    \subfloat[$P$]{%
        \includegraphics[width=0.3\textwidth]{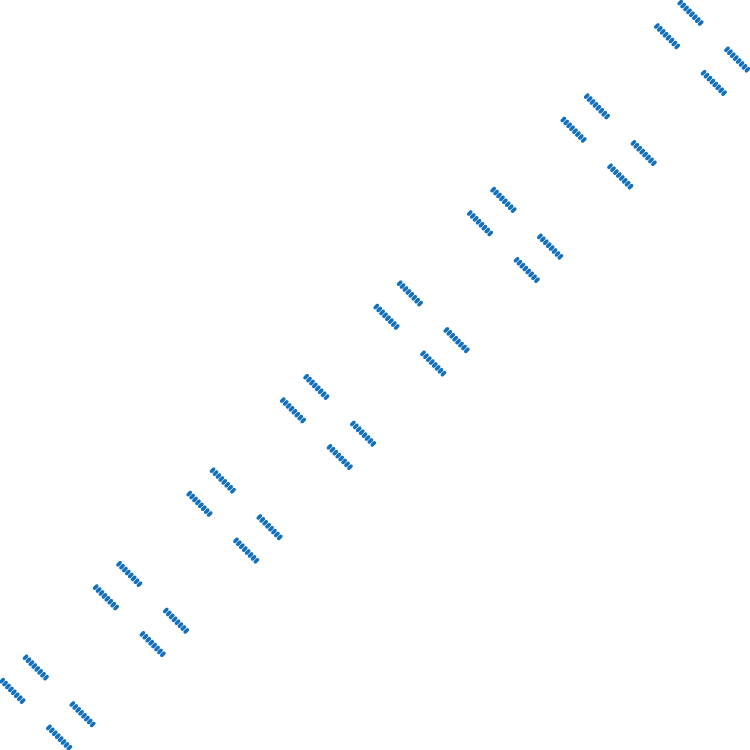}%
        }%
    \ 
    \subfloat[$L+U$]{%
        \includegraphics[width=0.3\textwidth]{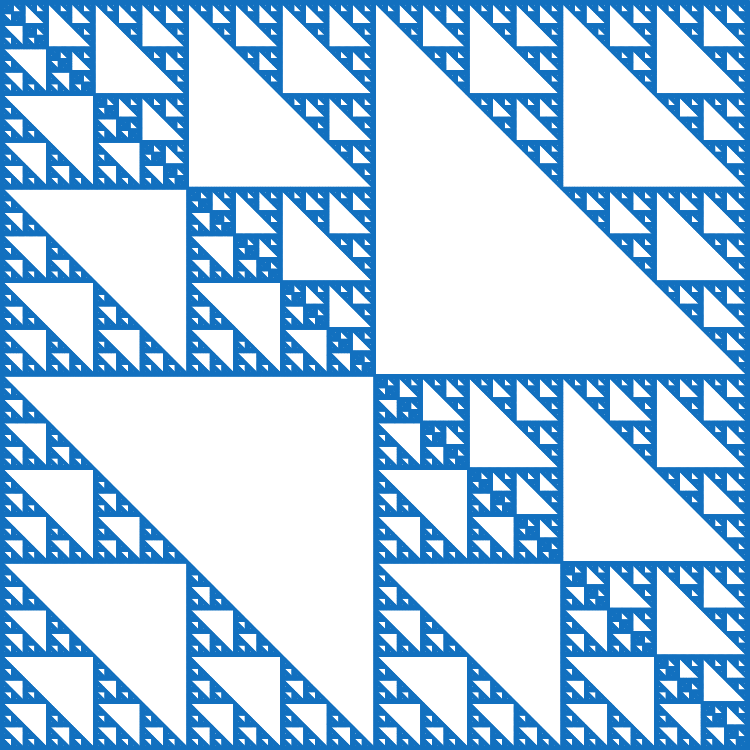}%
        }%
    \ \subfloat[$Q$]{%
        \includegraphics[width=0.3\textwidth]{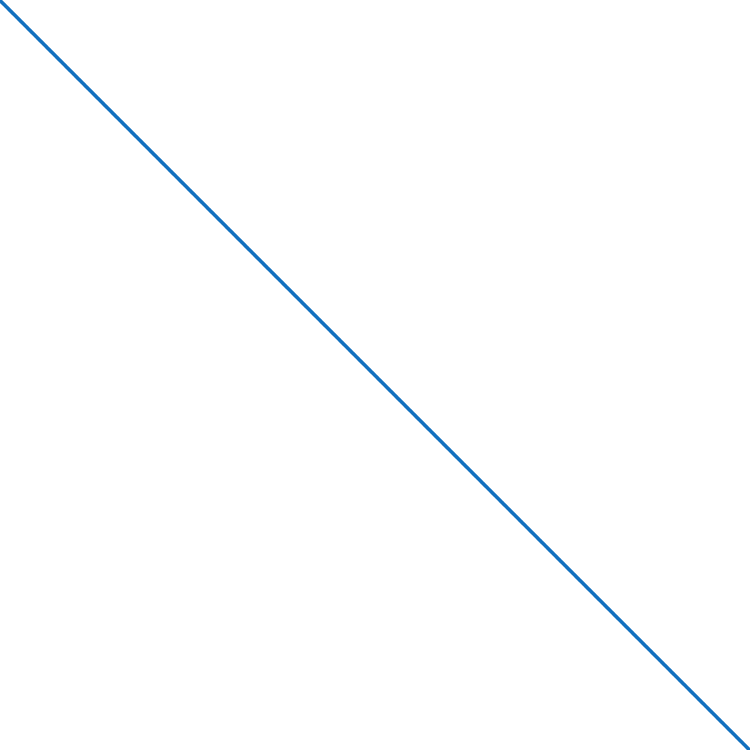}%
        }%
\caption{Sparsity patterns for (computed, within the given tolerance $\operatorname{tol} = 10^{\red 3} \cdot \epsilon_{\operatorname{machine}}$) matrix factors (a) $P$, (b) $L+U$, and (c) $Q$ from then GECP factorization $PBQ = LU$ for $B = B(\tilde{\boldsymbol{\t}}) \in \B_s(N)$ for $N = 2^{10}$ and $\boldsymbol{\t}\sim \Uniform([0,2\pi)^{10})$.}
\label{fig: lex CP}%
\end{figure}

Our goal is now to extend \Cref{thm: gf gepp} to include a subclass of butterfly matrices when using also GECP. In particular, we will introduce a class of matrices that  do not need any pivoting when using GECP, i.e., they are \textit{completely pivoted}:
\begin{theorem}\label{thm: new}
    Let $B(\boldsymbol{\t}) \in \B_s(N)$. If $\boldsymbol{\t} \in [0,2\pi)^n$ such that
    \begin{equation}\label{eq: thm 1st cond}
        |\tan\t_{j+1}| \le |\tan \t_j| \le 1
    \end{equation}
    for each $j$, then the GENP, GEPP, GERP, and GECP factorizations of $B({\boldsymbol{\t}})$ all align, i.e, $B(\boldsymbol{\t})$ is completely pivoted. In particular, then the growth factor takes the form \eqref{eq: gf angle vec}. Moreover, if $\boldsymbol{\t} \sim \Uniform([0,2\pi)^n)$ and $\sigma \in S_n$ is such that $\tilde \t_j = \t_{\sigma(j)}$ where
    \begin{equation}\label{eq: thm 2nd cond}
        \max(|\tan\tilde \t_{j+1}|,|\cot\tilde \t_{j+1}|) \le \max(|\tan \tilde \t_j|,|\cot \tilde \t_j|),
    \end{equation}
    then the GEPP, GERP, and GECP factorizations of $B(\tilde{\boldsymbol{\t}})$ all align and $\rho(B(\tilde{\boldsymbol{\t}}))$ has distribution determined by \Cref{thm: gf gepp}.
\end{theorem}

\Cref{fig: lex CP} shows plots of the sparsity patterns for the (computed) factors $P$, $L+U$, and $Q$ where $PBQ = LU$ is the GECP factorization for $B = B(\tilde{\boldsymbol{\t}}) \in \B_s(N)$ with $N = 2^{10}$, using $\boldsymbol{\t}\sim \Uniform([0,2\pi)^n)$. Note the $L$ and $U$ factors then have sparsity patterns determined by the Kronecker factor forms of $L$ and $U$, that then have nonzero entry locations aligning for $L$ and $U^\top$ in line with the Sierpi\'nski triangle, determined by $\bigotimes^n \begin{bmatrix}1 & 0 \\ 1 & 1\end{bmatrix}$; this follows directly from the mixed-product property and the $LU$ factorization
\begin{equation}\label{eq: LU rotation}
    R_\theta = \begin{bmatrix}
        1 & 0\\-\tan \t & 1
    \end{bmatrix} \begin{bmatrix}
        \cos \t & \sin \t \\ 0 & \sec \t
    \end{bmatrix} := L_\t U_\t.
\end{equation}
The $\max(|\tan \theta|,|\cot \theta|)$ terms in \Cref{thm: new} arise from the corresponding Kronecker factors in the GEPP factorization, together with the mixed-product property. In particular, since $B(\boldsymbol{\theta}) \in \B_s(N)$, its GEPP factorization is given by the Kronecker product of the GEPP factorizations of its constituent factors. Moreover, the condition $|\tan \theta| \ge 1$ (equivalently, $|\cos \theta| \le |\sin \theta|$) corresponds to the case where no GEPP pivot movement is required for $R_\theta$, while $|\tan \theta| < 1$ (equivalently, $|\cot \theta| > 1$) corresponds to when a pivot movement occurs.

Additionally, the map of $Q$ in \Cref{fig: lex CP} is the identity, indicating that no column pivots are required, while $P$ coincides with the simple butterfly permutation arising from GEPP applied to $B$. Properties of simple butterfly permutations, including those governing the number of GEPP pivot movements for matrices $B \sim \B_s(N,\Sigma_S)$, are studied in \cite{P24}, while uniform butterfly permutations --- corresponding to the GEPP-induced permutations for $\B(N,\Sigma_S)$ --- are analyzed in \cite{peca2024longest}. Future work will further investigate structural properties of the sparsity patterns of $P$ and $Q$, particularly from the perspective of permuton theory (cf. \cite{Borga21}).

\begin{figure}[t] 
\centering
    \subfloat[$P$]{%
        \includegraphics[width=0.3\textwidth]{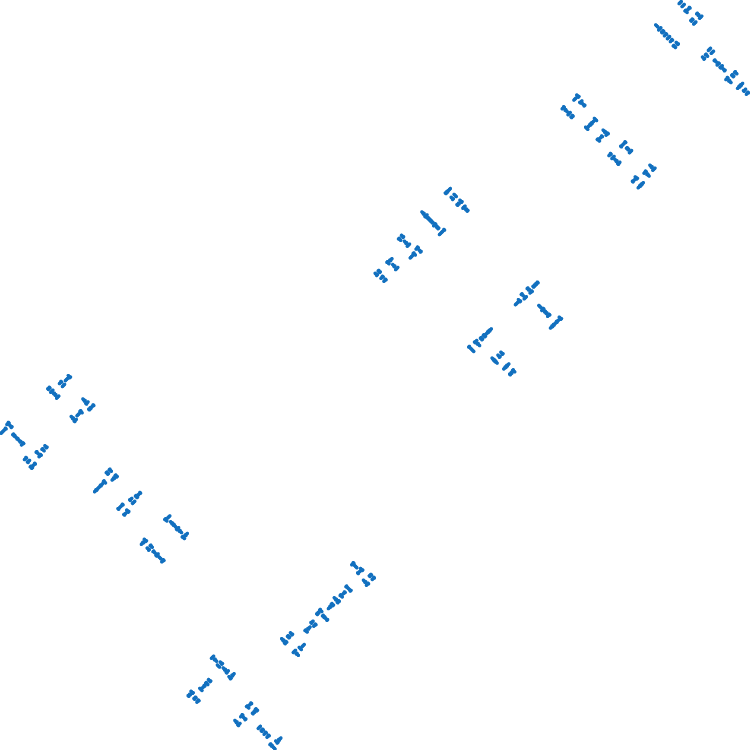}%
        }%
    \ 
    \subfloat[$L+U$]{%
        \includegraphics[width=0.3\textwidth]{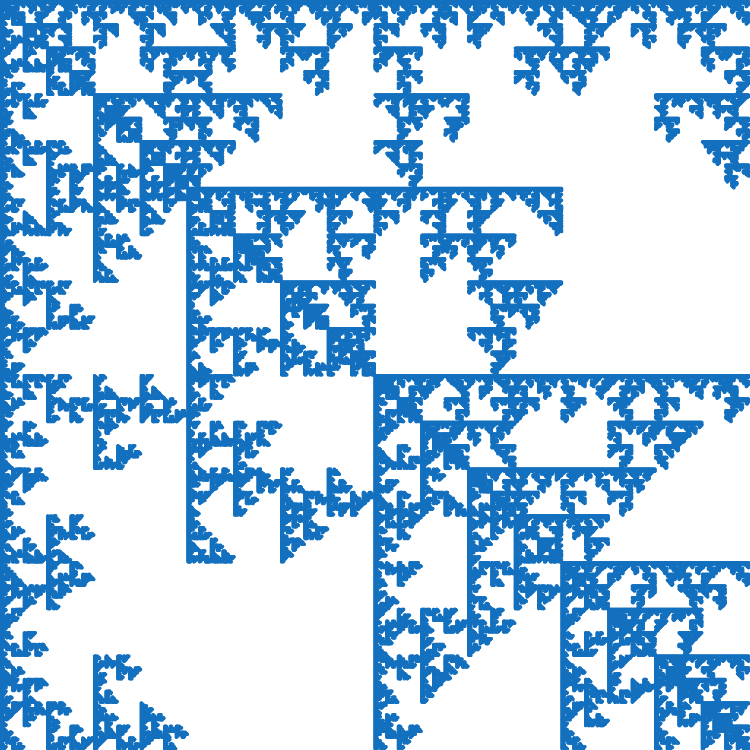}%
        }%
    \ \subfloat[$Q$]{%
        \includegraphics[width=0.3\textwidth]{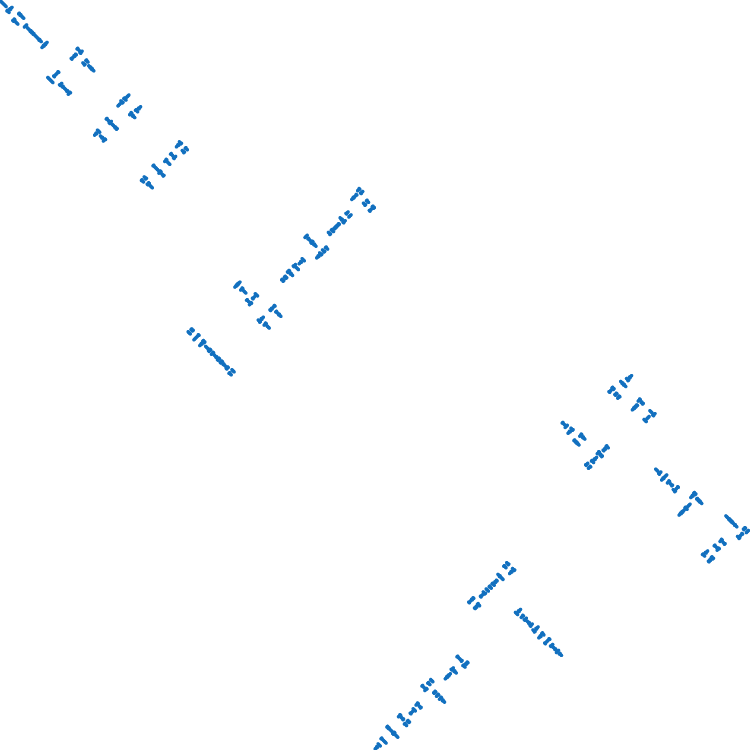}%
        }%
\caption{Sparsity patterns for the computed GECP factors of $B$ from \Cref{fig: lex CP}, now using GECP without an added tolerance parameter to identify potential candidates during each intermediate pivot search.}
\label{fig: lex CP notol}%
\end{figure}

\begin{remark}
    \Cref{fig: lex CP} used an additional tolerance level of $\operatorname{tol} = 10^{\red 3} \cdot \epsilon_{\operatorname{machine}}$ to determine the nonzero entries from $L + U$. This tolerance was also used for a custom MATLAB GECP algorithm to ensure the lexicographic column-major ordering was enforced, after first determining all possible pivot candidates within this tolerance level of the computed overall maximal entry in the remaining untriangularized block. \Cref{fig: lex CP notol} shows the sparsity patterns of the same $B$ from \Cref{fig: lex CP}, but now using a GECP algorithm without the added tolerance to determine the full list of potential pivot candidates. This in particular highlights how the floating-point arithmetic errors accumulate in such a way so that the computed ``maximal'' entry then manifests as a ``random'' candidate among what would be potential pivot candidates. Of note, although the overall sparsity pattern was not preserved without the tolerance parameter, the actual sparsity (i.e., proportion of zero entries, which still uses the tolerance to distinguish what a ``zero'' entry is) does remain intact. Additionally, the $L$ and $U^\top$ factors as well appear to maintain a symmetric sparsity pattern (as is confirmed at least with our sample since $L$ and $U^\top$ had entries with magnitude larger than the tolerance in the exact same locations).
\end{remark}

To prove \Cref{thm: new}, we will prove the following proposition that encapsulates the maximizing growth properties of the particular matrices considered in \Cref{thm: new}. This is an extension and update to a similar proposition used in \cite{PT23} to establish GEPP (and GERP) butterfly maximizing growth properties.

\begin{proposition}
\label{prop: gecp}
Let $B = B(\boldsymbol \t)  \in \B_s(N)$ such that $|\tan\t_{i+1}| \le |\tan\t_i| \le 1$ for all $i$, and let  $B = LU$ be the $LU$ factorization of $B$ using GENP. Let $\eta,\varepsilon \in \mathbb R$ such that 
$|\varepsilon| \le |\eta - \varepsilon|$. Then for all $k$
\begin{equation}
\label{eq: gecp max entry}
    \left\| \begin{bmatrix}
    \V 0 & \V I_{N-k+1}
    \end{bmatrix}(\eta B - \varepsilon B^{(k)})\begin{bmatrix}
    \V 0 \\ \V I_{N-k+1}
    \end{bmatrix} \right\|_{\max} \le |\eta - \varepsilon||U_{kk}|.
\end{equation}
In particular, then $B = LU$ is also the $LU$ factorization of $B$ using GECP, i.e., \Cref{thm: new} holds.
\end{proposition}

\red{We first establish a useful technical lemma.}
\begin{lemma}
\label{lemma: eta eps ineq}
\red{Let $\alpha,\beta,\t,\eta,\varepsilon \in \mathbb R$ such that $|\varepsilon| \le |\eta - \varepsilon|$. Then
\begin{equation*}
    |\eta\alpha \cos^2\t  - \varepsilon \beta| \le \cos^2\t|\eta \alpha - \varepsilon \beta| + \sin^2\t |\eta - \varepsilon||\beta|
\end{equation*}}
\end{lemma}
\begin{proof}
\red{Since $
\eta\alpha \cos^2\t  - \varepsilon  \beta = \cos^2\t(\eta \alpha - \varepsilon \beta) - \sin^2\t\varepsilon \beta,
$ this follows directly by the triangle inequality and $|\varepsilon| \le |\eta - \varepsilon|$.}
\end{proof}
\red{We can now prove \Cref{prop: gecp}.}
\begin{proof}[Proof of \Cref{prop: gecp}]
First, note the last implication \red{that $B$ is completely pivoted} follows since then 
$$
|B^{(k)}_{kk}| \le \left\|\begin{bmatrix}
    \V 0 & \V I_{N-k+1}
\end{bmatrix}B^{(k)} \begin{bmatrix}
    \V 0 \\ \V I_{N-k+1}
\end{bmatrix}\right\|_{\max} \le |U_{kk}| = |B^{(k)}_{kk}|
$$
for all $k$ using \eqref{eq: gecp max entry} with $\eta = 0$ and $\varepsilon = 1$, so that no row or column swaps would be needed at any intermediate GECP step. This determines the first part of \Cref{thm: new}, while the latter half follows then also from \Cref{thm: gf gepp} and \Cref{cor: gf inv sym}.

To prove \eqref{eq: gecp max entry}, we will use induction on $n = \log_2 N$. Note first the result is immediate for $k=1$ since then 
$$
\|\V I_N (\eta B - \varepsilon B^{(1)}) \V I_N\|_{\max} = |\eta - \varepsilon| \|B\|_{\max} = |\eta-\varepsilon||U_{11}|.
$$ So it suffices to assume $k \ge 2$. If $n = 1$, then using \eqref{eq: LU rotation} with  $B^{(2)} = U = U_\t$, we see for $k = 2$ then
$$
\left\| \begin{bmatrix}
0 & 1
\end{bmatrix} (\eta B - \varepsilon B^{(2)})\begin{bmatrix}
0 \\ 1
\end{bmatrix}\right\|_{\max} = |\eta \cos\t - \varepsilon \sec\t| = |\eta \cos^2\t - \varepsilon| |U_{22}| \le |\eta - \varepsilon| |U_{22}|
$$
using \Cref{lemma: eta eps ineq} with $\alpha = \beta = 1$.

Now assume the result holds for $\B_s(N/2)$ and $n \ge 2$, and let $B = B(\t,A) \in \B_s(N)$ for $A = B(\boldsymbol \t') \in \B_s(N/2)$ such that $\boldsymbol \t = (\boldsymbol \t',\t)$ with $\t=\t_n$, where we note also $\boldsymbol\t'$ still satisfies the condition $|\tan\t_{i+1}'| \le |\tan\t_i'| \le 1$ for all $i$. Since $B$ has an $LU$ factorization using GENP, then necessarily $A$ does also. Let $A = L'U'$ be this factorization, where we further note by \Cref{lemma:Bk} and \eqref{eq: LU rotation} then 
\begin{equation}
\label{eq: gecp uk form}
    U_{kk} = \left\{
    \begin{array}{ll}
    \cos\t U_{kk}' & \mbox{if $k \le N/2$}\\
    \sec\t U_{jj}' & \mbox{if $k = N/2 + j$ for $j\ge 1$.}
    \end{array}
    \right.
\end{equation}

For $1<k \le N/2$, then for $\V I = \V I_{N/2-k+1}$ when not indicated otherwise, we have
\begin{align*}
    &\begin{bmatrix}
    \V 0  &\V I_{N-k+1}
    \end{bmatrix}(\eta B - \varepsilon B^{(k)})\begin{bmatrix}
    \V 0 \\ \V I_{N-k+1}
    \end{bmatrix} \nonumber\\
    &\hspace{1pc}= \begin{bmatrix}
    \cos\t\begin{bmatrix}
    \V 0 & \V I
    \end{bmatrix}(\eta A - \varepsilon A^{(k)})\begin{bmatrix}
    \V 0 \\ \V I
    \end{bmatrix} & \sin\t\begin{bmatrix}
    \V 0 & \V I
    \end{bmatrix}(\eta A - \varepsilon A^{(k)})\\
    -\sin\t\left(\eta A - \varepsilon \begin{bmatrix}
    \V0\\&\V I
    \end{bmatrix}A^{(k)}\right)\begin{bmatrix}
    \V 0 \\ \V I
    \end{bmatrix}
    &- \sec\t\left((\varepsilon-\eta\cos^2\theta)A - \varepsilon\sin^2\t  \begin{bmatrix}
    \V0\\&\V I
    \end{bmatrix}A^{(k)}\right)
    \end{bmatrix} \label{eq: etaB-epBk 1}
\end{align*}
using \Cref{lemma:Bk}. Let $\eta'=\varepsilon - \eta \cos^2\t$ and $\varepsilon' =  \varepsilon \sin^2\t$ where we note 
$$
|\varepsilon'| = \sin^2\t|\varepsilon| \le \cos^2\t|\eta - \varepsilon| = |\eta' - \varepsilon'|
$$
since $\sin^2\t\le \cos^2\t$ and $|\varepsilon|\le|\eta-\varepsilon|$. By the inductive hypothesis, we have
\begin{align*}
\left\|\begin{bmatrix}
    \V 0  &\V I
    \end{bmatrix}(\eta' A - \varepsilon' A^{(k)})\begin{bmatrix}
    \V 0 \\ \V I
    \end{bmatrix}\right\|_{\max} 
    &\le |\eta'-\varepsilon'||U'_{kk}| 
    = |\cos\t||\eta - \varepsilon||U_{kk}| \quad \mbox{and}\\
\left\|\begin{bmatrix} \V0 &\V I
\end{bmatrix}(\eta A - \varepsilon A^{(k)}) \begin{bmatrix}
\V0 \\\V I
\end{bmatrix} \right\|_{\max} &\le |\eta - \varepsilon||U_{kk}'| = |\sec\t||\eta - \varepsilon||U_{kk}|
\end{align*}
using \eqref{eq: gecp uk form}. Also, since $|\cos\t_{i+1}|\ge |\cos\t_i|$ for all $i$, then
$$
|\sec\t|\|A\|_{\max} = |\sec\t_n||U_{11}'| = \frac{\prod_{j=1}^{n-1}|\cos\t_j|}{|\cos\t_n|} \le \frac{\prod_{j \in \mathcal J_k}|\cos\t_j|}{\prod_{j \in [n] \backslash \mathcal J_k} |\cos\t_j|} =  |U_{kk}|
$$
for some $\varnothing \ne \mathcal J_k \subset [n]$ when $k>1$ by \Cref{lemma:Bk}. Next, we see
$$
|\eta \sin\t|\|A\|_{\max} \le |\eta - \varepsilon| |2\sin\t| \|A\|_{\max} \le |\eta - \varepsilon| |\sec\t|\|A\|_{\max} \le |\eta - \varepsilon||U_{kk}|
$$
using $|\eta|\le |\eta - \varepsilon|+|\varepsilon| \le 2|\eta - \varepsilon|$ and $|2\sin\t| \le |\sec\t|$ (since $|2\sin\t\cos\t| = |\sin(2\t)| \le 1)$. It follows 
\begin{align*}
&\left\|\begin{bmatrix}
    \V 0  &\V I_{N-k+1}
    \end{bmatrix}(\eta B - \varepsilon B^{(k)})\begin{bmatrix}
    \V 0 \\ \V I_{N-k+1}
    \end{bmatrix}\right\|_{\max}\\
    &\hspace{1pc} =\max\left(
    \begin{array}{c}
    |\eta \sin\theta|\|A_{N/2-k+1:,:k-1}\|_{\max},\\
    |\eta \sin\theta|\|A_{:k-1,N/2-k+1:}\|_{\max},\\
    |\cos\t|\left\|\begin{bmatrix} \V0 &\V I
\end{bmatrix}(\eta A - \varepsilon A^{(k)}) \begin{bmatrix}
\V0 \\\V I
\end{bmatrix} \right\|_{\max},\\
|\sec\t||\varepsilon - \eta \cos^2\theta|\| A\|_{\max},\\
|\sec\t|\left\|\begin{bmatrix}
    \V0&\V I
    \end{bmatrix}\left((\varepsilon-\eta \cos^2\theta )A - \varepsilon\sin^2\t  A^{(k)}\right)\begin{bmatrix}
    \V0\\\V I
    \end{bmatrix} \right\|_{\max}
\end{array}\right)\\
&\hspace{1pc}\le |\eta - \varepsilon||U_{kk}|
\end{align*}
using again \Cref{lemma: eta eps ineq} so that $|\varepsilon - \eta \cos^2\t| \le |\eta - \varepsilon|$.

For $k = N/2+j$ and $j \ge 1$ so that $N-k+1=N/2-j+1$, writing now $\V I = \V I_{N-k+1} = \V I_{N/2-j+1}$, we have
\begin{align*}
&\left\|\begin{bmatrix}
    \V 0  &\V I
    \end{bmatrix}(\eta B - \varepsilon B^{(k)})\begin{bmatrix}
    \V 0 \\ \V I
    \end{bmatrix}\right\|_{\max}\\
    &\hspace{1pc}= |\sec\t|\left\|\begin{bmatrix}
    \V 0  &\V I
    \end{bmatrix}(\eta \cos^2\t A - \varepsilon  A^{(j)})\begin{bmatrix}
    \V 0 \\ \V I
    \end{bmatrix} \right\|_{\max}\\
    &\hspace{1pc}= |\sec\t|\left\|\begin{bmatrix}
    \V 0  &\V I
    \end{bmatrix}\left|\eta \cos^2\t A - \varepsilon  A^{(j)}\right|\begin{bmatrix}
    \V 0 \\ \V I
    \end{bmatrix} \right\|_{\max}\\
    &\hspace{1pc}\le |\sec\t|\left(\cos^2\t\left\|\begin{bmatrix}
    \V 0  &\V I
    \end{bmatrix}\left|\eta  A - \varepsilon  A^{(j)}\right|\begin{bmatrix}
    \V 0 \\ \V I
    \end{bmatrix} \right\|_{\max}+\sin^2\t|\eta - \varepsilon| \left\|\begin{bmatrix}
        \V 0 & \V I
    \end{bmatrix} |A^{(j)}|\begin{bmatrix}
        \V 0 \\ \V I
    \end{bmatrix}\right\|_{\max}\right)\\
    &\hspace{1pc}= |\sec\t|\left(\cos^2\t\left\|\begin{bmatrix}
    \V 0  &\V I
    \end{bmatrix}(\eta  A - \varepsilon  A^{(j)})\begin{bmatrix}
    \V 0 \\ \V I
    \end{bmatrix} \right\|_{\max}+\sin^2\t|\eta - \varepsilon| \left\|\begin{bmatrix}
        \V 0 & \V I
    \end{bmatrix} A^{(j)}\begin{bmatrix}
        \V 0 \\ \V I
    \end{bmatrix}\right\|_{\max}\right)\\
    &\hspace{1pc}\le |\sec\t||\eta - \varepsilon||U'_{jj}| = |\eta - \varepsilon||U_{kk}|
\end{align*}
using \Cref{lemma:Bk} for the first equality, \Cref{lemma: eta eps ineq} for the first inequality (applied componentwise with $\alpha = A_{i'j'}$ and $\beta = A^{(j)}_{i'j'}$), the inductive hypothesis for the last inequality (with $\eta = 0$ and $\varepsilon = 1$ for the second term), \eqref{eq: gecp uk form} for the last equality, and the fact $\|A\|_{\max} = \| |A| \|_{\max}$ for the remaining steps.
\end{proof}

\Cref{thm: new} extends a classic result from Tornheim from 1970 that produced particular Sylvester Hadamard matrices that were complete pivoted:
\begin{proposition}[\cite{T70}]\label{prop: torn}
    If $A$ is completely pivoted and $H = \begin{bmatrix}
        1 & 1 \\ -1 & 1
    \end{bmatrix}$, then $A \otimes H$ is completely pivoted. In particular, then $\bigotimes^n H$ is completely pivoted.
\end{proposition}
There is a natural relationship between butterfly matrices and particular Sylvester Hadamard matrices, which will be further outlined in \Cref{sec: had}. Of note, butterfly matrices can be interpreted as continuous generalizations of these Sylvester Hadamard matrices, on which these Hadamard matrices then achieve certain maximizing properties for particular statistics, including the growth factors. \Cref{prop: torn} only guarantees $B({\boldsymbol{\t}}) \otimes H$ remains completely pivoted if $B(\boldsymbol{\t}) \in \B_s(N)$ is completely pivoted. Since $H = \sqrt 2 \cdot B(\frac\pi4,\frac\pi4)$, and matrices remain completely pivoted under multiplication by scalar matrices and diagonal sign matrices (e.g., if $A \in \mathbb R^{n \times n}$ is completely pivoted, then $A\cdot (\bigoplus^n \pm 1)$ is also completely pivoted), then \Cref{thm: new} further guarantees $B(\boldsymbol{\t}) \otimes R_\varphi$ is completely pivoted if $|\tan \varphi| \ge \max_j |\tan \t_j|$ rather than requiring only $|\tan \varphi| = 1$. Moreover, as noted in \cite{DaPe88}, \Cref{prop: torn} is not symmetric: if $A$ is completely pivoted, then $H \otimes A$ is not necessarily completely pivoted; this can be seen by considering $A = \begin{bmatrix}
    1 & 1 \\ 1 & 1
\end{bmatrix}$. \Cref{thm: new} has the additional flexibility that $R_\varphi \otimes B(\boldsymbol{\t})$ is completely pivoted if $|\tan \varphi| \le \min_j |\tan \t_j|$, as well as forming additionally completely pivoted butterfly matrices by embedding Kronecker rotation matrix factors that preserve the ordering from \Cref{thm: new}.

\subsection{Experiments for computed permutation factors}
\label{sec: experiments}

A natural question from \Cref{thm: new} is whether the additional structure of $\B_s(N)$ lends itself so that GECP factorizations are attainable using \textit{any} ordering of input angles. Unfortunately, this happens to only be the case  for small $n$.

\red{Let $Q_n$ denote} the perfect shuffle such that $Q_n(A \otimes B)Q_n^\top = B \otimes A$ for $A \in \mathbb R^{N/2 \times N/2}$ and $B \in \mathbb R^{2 \times 2}$. For example, $Q_1 = \V I_2$ while $Q_2 = P_{(2 \ 3)} \in \mathbb R^{4 \times 4}$. Moreover, all possible such perfect shuffles $\tilde Q_n$ are multiplicatively generated by $\V I_{2^i} \otimes Q_2 \otimes \V I_{2^j}$ where $i + j = n - 2$, as these perfect shuffles themselves are isomorphic to $S_n$, which is similarly generated by the transpositions $(i \ i+1)$. So the question then for $B \in \B_s(N)$, is when is $PBQ = LU$ the GECP factorization with $P = P_{\tilde B}\tilde Q_n$ and $Q = \tilde Q_n^\top$, where $P_{\tilde B} \tilde B = L_1U_1$ is the GEPP factorization of $B$ while $\tilde Q_n B \tilde Q_n^\top$ satisfies \eqref{eq: thm 1st cond}?

\Cref{t: perfect align} shows the output using computed GECP factorizations (using the added tolerance parameter for an expanded pivot candidate search) for $B \sim \B_s(N,\Sigma_S)$, counting only those where the reordered input angles produce a permutation factor other than those formed using the perfect shuffles $\tilde Q_n$ and the GEPP butterfly permutation. (See \cite{DAngeli_Donno_2017,P24} for the particular perfect shuffle and butterfly permutation forms then possible.) The set of experiments used multiple samples first of $\boldsymbol{\t} \sim \Uniform([0,2\pi)^n)$, transforming each angle with $\hat \t = \t$ if $|\tan \t| \le 1$ (i.e., no GEPP pivot movement is needed on $R_\t$) or $\hat \t = \frac\pi2 - \t$ if $|\tan \t| > 1$ (i.e., a GEPP pivot movement is needed on $R_\t$), then reordering the transformed angles $\hat{\boldsymbol{\t}}$ into $\tilde{\boldsymbol{\t}}$ so that $B(\tilde{\boldsymbol{\t}})$ is completely pivoted (i.e., satisfies the hypothesis of \Cref{thm: new}). Then the GECP factorization (using the added tolerance of $10^3 \cdot \epsilon_{\operatorname{machine}}$ to determine an initial list of potential pivot candidates) is computed for each of the $|S_n| = n!$ rearrangements of the angles of $\tilde{\boldsymbol{\t}}$. \Cref{t: perfect align} then records the count per $\sigma \in S_n$ rearrangement $\sigma(\boldsymbol{\t})$ such that $\sigma(\boldsymbol{\t})_j = \tilde{\boldsymbol{\t}}_{\sigma(j)}$ of $\tilde{\boldsymbol{\t}}$ only if either the computed $P,Q$ factors differ from the corresponding perfect shuffle $\tilde Q_n$ to return $PB(\sigma(\tilde{\boldsymbol{\t}}))Q = B(\tilde{\boldsymbol{\t}})$. The experiment is then repeated (between 2 to 5 times) for differing initial $\boldsymbol{\t} \sim \Uniform([0,2\pi)^n)$, recording any instance where a computed GECP factorization does not align with the corresponding perfect shuffle rearrangement.

\begin{table}[ht!]
\centering
{
\begin{tabular}{r|ccc}
$n$ & $|S_n|$ & $\#\{PBQ \ne (\tilde Q_n)B(\tilde Q_n^\top)\}$ & \%\\ \hline 
1 & 1& - & -\\
2 & 2 & - & -\\
3 & 6 & - & -\\
4 & 24 & 2 & 8.3\%\\
5 & 120 & 36 & 30.0\%\\
6 & 720 & 381 & 52.9\%
\end{tabular}
}
\caption{Computed number of GECP factorizations $PBQ = LU$ for $B \in \B_s(N)$ where $P,Q$ take a form other than that expected from the corresponding perfect shuffle and GEPP factorization.}
\label{t: perfect align}
\end{table}

It happens that for $n = 1,2,3$, then the only permutation factors encountered are precisely those of the form $P = \red{P_{\tilde B}}\tilde Q_n$ and $Q = \tilde Q_n^\top$. \red{This is immediate for $n = 1$, since $Q_1 = \mathbf I_2$, so no column permutations occur. In this case, $P = \mathbf I_2$ if $|\tan \theta| \le 1$ and $P = P_{(1\ 2)}$ otherwise, 
which exactly matches the GEPP permutation $P_{\tilde B}$, and hence is consistent with 
the form $P = P_{\tilde B} \tilde Q_1$, $Q = \tilde Q_1^\top$ with $\tilde Q_1 = \mathbf I_2$. For $n = 2$, we illustrate the alignment in the simplest nontrivial case where $|\tan \t_j| \le 1$ for each $j$ (so no GEPP pivot movements are needed), noting that more detailed computations quickly become unwieldy, although they remain consistent with the observed behavior. Here $Q_2 = P_{(2\ 3)}$ and $Q_2 B(\theta_1,\theta_2) Q_2 = B(\theta_2,\theta_1)$, while $S_2 = \{1,(1\ 2)\}$.} 

\red{If $\sigma = 1$, then $|\tan \t_2| \le |\tan \t_1|$, so $B = B(\t_1,\t_2)$ is completely pivoted by \Cref{thm: new}. In this case, the GEPP, GECP, and GENP factorizations coincide, and hence $P = Q^\top = \tilde Q_2 = \V I_4$, yielding $PBQ = B = \tilde Q_2 B \tilde Q_2^\top$.} 

\red{If $\sigma = (1\ 2)$, then $|\tan \t_2| > |\tan \t_1|$, so that $B(\t_2,\t_1) = Q_2 B(\t_1,\t_2) Q_2$ is completely pivoted, and hence $\tilde Q_2 = Q_2$. By \Cref{thm: gf gepp}, $B(\t_1,\t_2)$ has aligning GEPP and GENP factorizations since $|\tan \t_j| \le 1$ for each $j$.} 

\red{In this setting, applying GECP to $B = B(\t_1,\t_2)$ has $\|B\|_{\max} = |B_{11}| = |\cos\t_1 \cos\t_2|$, so no GECP pivot movements are needed on the first step. After one elimination step, the remaining untriangularized block is
$$
B^{(2)}_{2:4,2:4} =
\begin{bmatrix}
\cos\t_2\sec\t_1 & 0 & \sin\t_2\sec\t_1\\
0 & \sec\t_2\cos\t_1 & \sec\t_2\sin\t_1\\
-\sin\t_2\sec\t_1 & -\sec\t_2\sin\t_1 & \sec\t_2\sec\t_1(\cos^2\t_1-\sin^2\t_2)
\end{bmatrix}
$$
using \Cref{lemma:Bk}. A direct comparison shows $\|B^{(2)}_{2:4,2:4}\|_{\max} = |B^{(2)}_{3,3}| = |\sec\t_2 \cos\t_1|$. To establish this, under the ordering $1 \ge |\tan \t_2| > |\tan \t_1|$, we have $|\sin\t_2| \le |\cos\t_2|$ and $|\sin\t_1| \le |\cos\t_1|$, so all off-diagonal terms are bounded by the diagonal contribution. Moreover,
\[
|\cos^2\t_1 - \sin^2\t_2| \le \cos^2\t_1
\]
since $\sin^2 \t_2 \le \cos^2\t_2 \le \cos^2\t_1 \le 2\cos^2\t_1$, which yields that the mixed term in the $(4,4)$-entry is also dominated. Thus the pivot sequence agrees with that induced by the perfect shuffle, yielding $P^{(2)} = (Q^{(2)})^\top = P_{(2\ 3)} = Q_2$. Since $Q_2 B Q_2^\top$ is completely pivoted, it follows that $P = Q^\top = Q_2 = \tilde Q_2$, and hence $PBQ = \tilde Q_2 B \tilde Q_2^\top$.}

\red{The remaining cases for $n=2$ (e.g., $|\tan \t_2| \ge 1 \ge |\tan \t_1|$ and $|\tan \t_2| \ge |\tan \t_1| \ge 1$) can be analyzed similarly (where now $P_{\tilde B}$ is a Kronecker product of powers of $P_{(1 \ 2)}$ by \Cref{thm: gf gepp}). The case $n=3$ follows from analogous but much more involved computations, which we omit for brevity.} 

However, \red{this structural alignment no longer persists} starting at $n = 4$. As $n$ increases, the proportion of \red{angle configurations that align with the perfect shuffle permutation decreases, and by $n = 6$ at least half of all rearrangements differ from the perfect shuffle.}

\begin{figure}[t] 
\centering
    \subfloat[$P$]{%
        \includegraphics[width=0.3\textwidth]{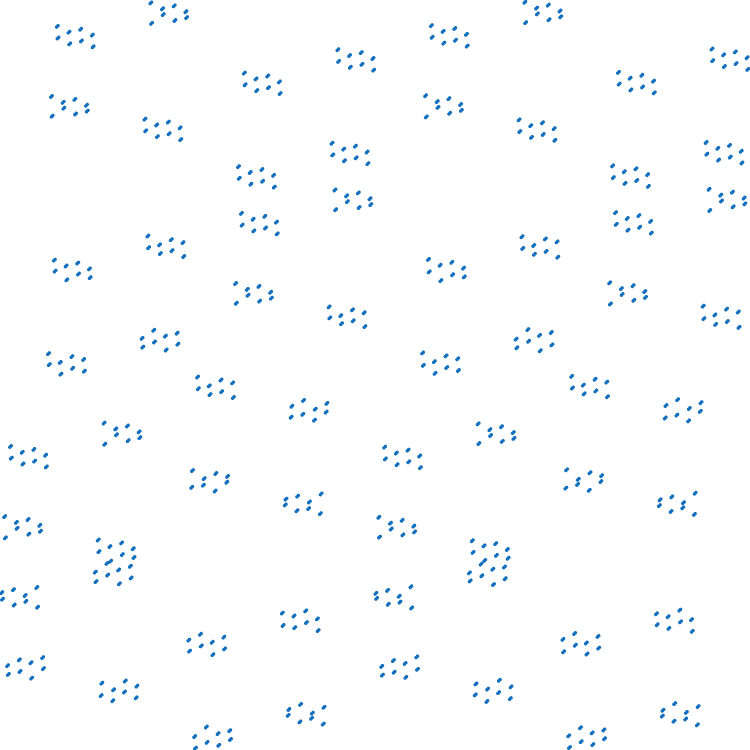}%
        }%
    \ 
    \subfloat[$L+U$]{%
        \includegraphics[width=0.3\textwidth]{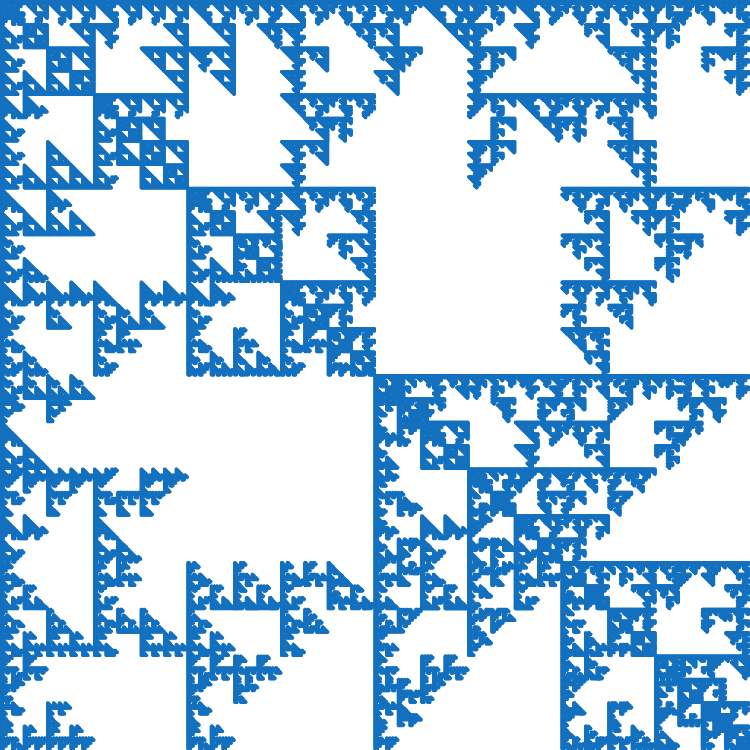}%
        }%
    \ \subfloat[$Q$]{%
        \includegraphics[width=0.3\textwidth]{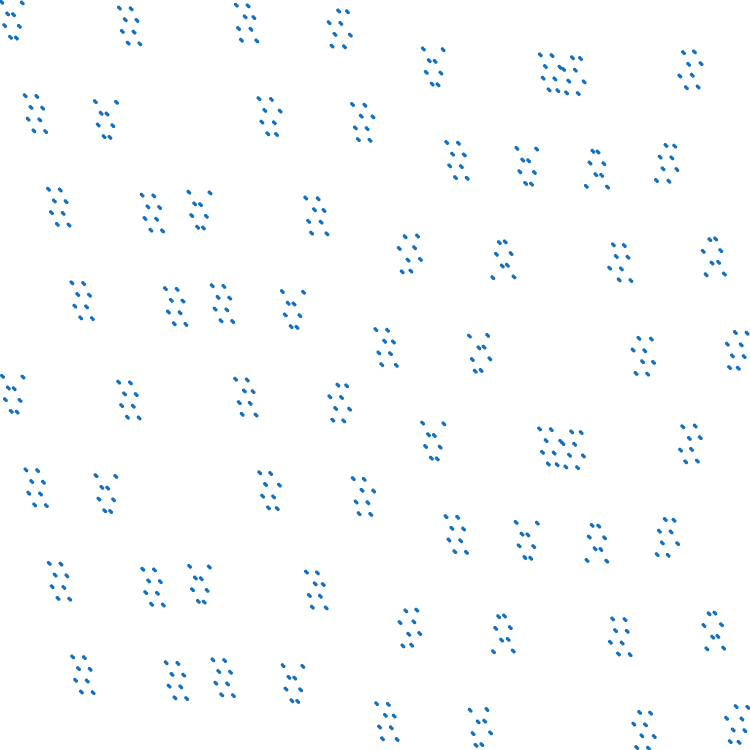}%
        }%
        \\
    \subfloat[$P$]{%
        \includegraphics[width=0.3\textwidth]{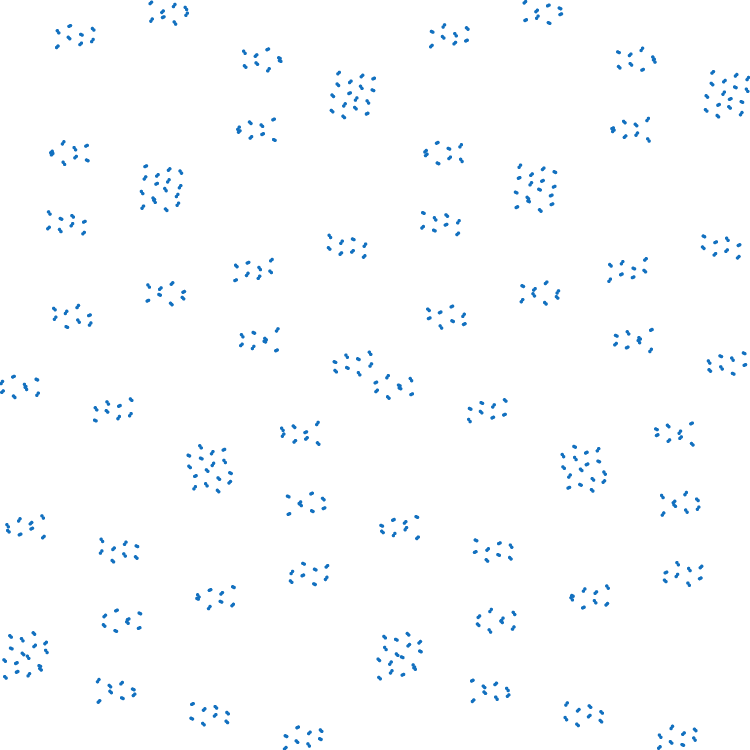}%
        }%
    \ 
    \subfloat[$L+U$]{%
        \includegraphics[width=0.3\textwidth]{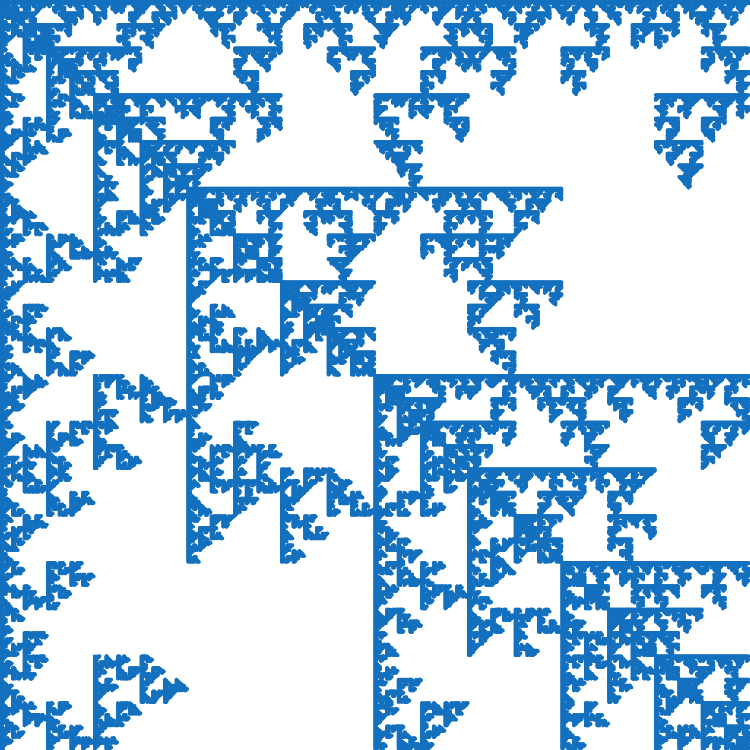}%
        }%
    \ \subfloat[$Q$]{%
        \includegraphics[width=0.3\textwidth]{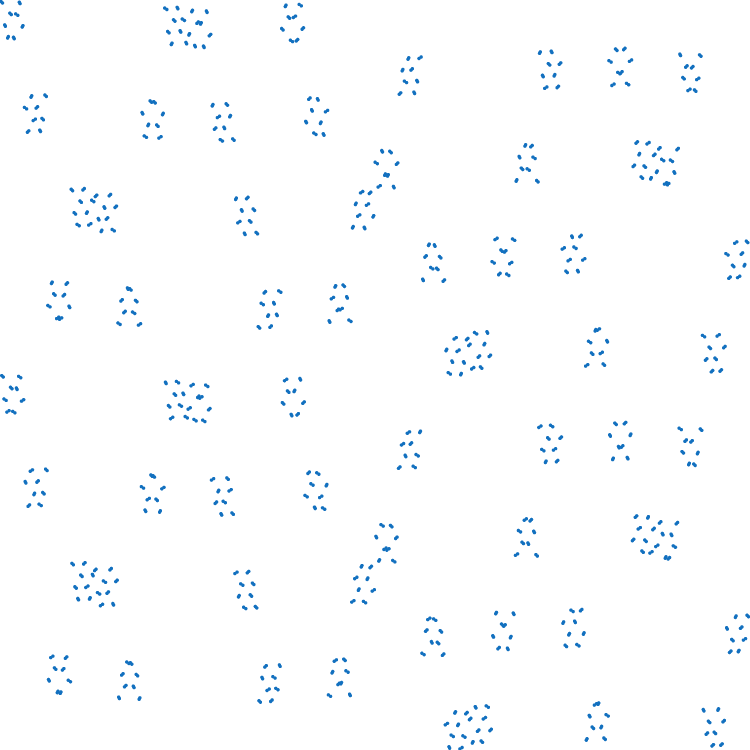}%
        }%
\caption{Sparsity patterns for the computed GECP factors of $B \in \B_s(2^{10})$ not satisfying \Cref{thm: new}, with (a)-(c) corresponding to GECP with an added tolerance parameter during the pivot candidate compilation, and (d)-(f) without the added tolerance setting.}
\label{fig: not CP}%
\end{figure}

For example, if $\boldsymbol{\t} \in [0,2\pi)^4$ such that $|\tan \t_4| < |\tan \t_3| < |\tan \t_2| < |\tan \t_1| \le 1$, then the computed GECP factorization using the reordered input angles $\boldsymbol{\t}=[\t_4,\t_3,\t_1,\t_2]^\top$ takes the form $P_{\sigma_1}B(\boldsymbol{\t})P_{\sigma_2}^\top$, where 
\begin{equation*}
    \sigma_1 = (12 \ 14)(10 \ 15)(8 \ 10)(7 \ 9)(6 \ 14)(5 \ 6)(4 \ 13)(3 \ 5)(2 \ 9) = \sigma_2^{-1}
\end{equation*}
compared to the  perfect shuffle permutation $\rho$ with $\tilde Q_n = P_\rho$, where
\begin{equation*}
    \rho = (14 \ 15)(12 \ 15)(10 \ 11)(8 \ 14)(7 \ 10)(6 \ 10)(5 \ 9)(4 \ 13)(3 \ 5)(2 \ 9).
\end{equation*}
In particular, considering the given forms \eqref{eq: perm form} for each permutation, then the first GECP step where the computed permutation differed from the perfect shuffle was at step 5. (So the first 4 GECP steps align for both strategies.) Two pivot candidates at GECP step 5 include the entries at locations $(6,6)$ and $(9,9)$. Using the column-major lexicographic GECP tie-breaking strategy, then the $(6,6)$ is deemed the new pivot and so each computed permutation uses the row and column transpositions $(5\ 6)$. In order to have recovered the necessary perfect shuffle, the GECP tie-breaking strategy would have needed to choose instead the entry at $(9,9)$ and hence the transposition $(5 \ 9)$.

\Cref{fig: not CP} shows both the sparsity patterns for the GECP factors of $B$ not in the particular ordered form found in \Cref{thm: new}, first showing the GECP with the added tolerance parameter (to widen the potential pivot candidates before implementing the column-major lexicographic tie-breaking strategy), which shows the Kronecker product is broken for this particular input $B \in \B_s(2^{10})$. This also includes the sparsity patterns for the GECP factors without using the added tolerance parameter for the pivot search, that essentially returns a ``random'' pivot among the potential pivot candidates due to rounding. Again, an interesting observation is the sparsity and symmetry of the final $L+U$ terms is preserved still, both with and without the added tolerance.

\begin{figure}[t] 
\centering
    \subfloat[$\rho(B)$]{%
        \includegraphics[width=0.45\textwidth]{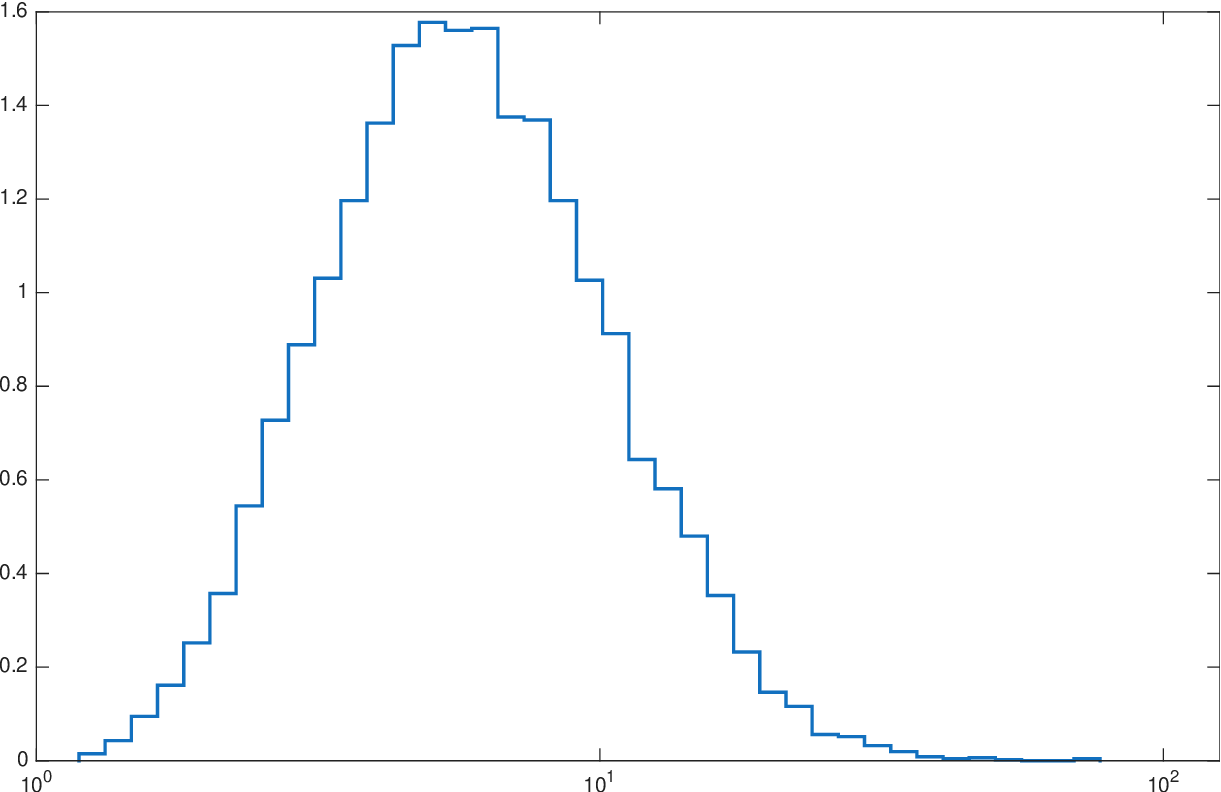}%
        }%
    \ 
    \subfloat[$\rho_\infty(B)$]{%
        \includegraphics[width=0.45\textwidth]{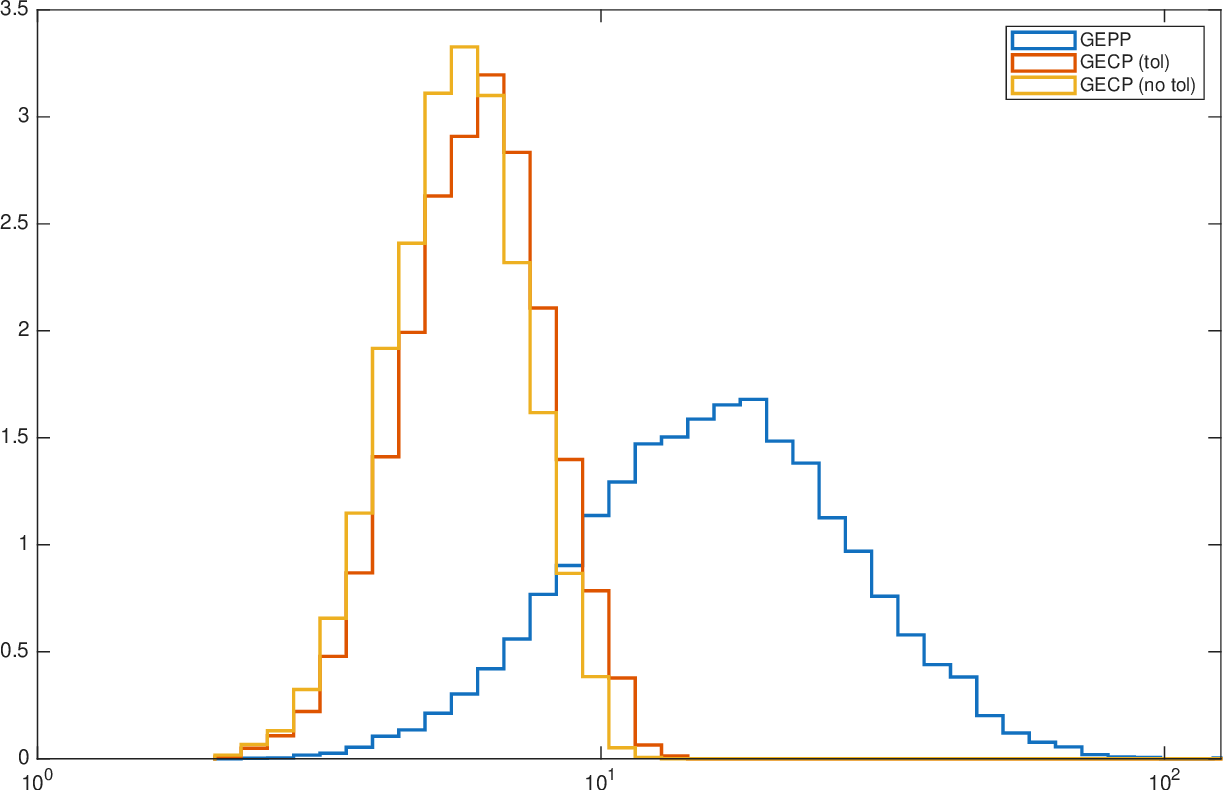}%
        }%
\caption{Histograms of computed growth factors (a) $\rho(B)$ and (b) $\rho_\infty(B)$ where $B \sim \B_s(N,\Sigma_S)$, using GEPP and both GECP with and without an added tolerance parameter to enforce the column-major lexicographic tie-breaking strategy, using $10^4$ trials with $N = 2^8$ and logarithmic scaling. Each setting yielded indistinguishable histograms for $\rho(B)$ in (a), so the legend is omitted. The $\rho_\infty$ growth factor in (b) shows sensitivity to the pivoting strategy and tolerance setting.}
\label{fig: compare hist}%
\end{figure}

Although \Cref{thm: new} yields the explicit random growth factors for butterfly matrices whose input angles preserve a particular monotonicity property, we expect the random growth factors using GECP to not actually need to rely on the input ordering. \Cref{thm: new} used this ordering to then determine every explicit intermediate GECP form of the butterfly linear system, including explicitly the final matrix factorization $PBQ = LU$, that could then determine how the local growth properties are presented at each step. \Cref{thm: gf gepp} guaranteed the growth factor using GE with \textit{any} pivoting scheme is minimized using GEPP. This further established that GERP also attained this minimal value using any ordering of the input angles, as the GERP and GEPP factorizations then aligned perfectly (still assuming a column-major GERP pivot search). 

Empirically, we expect this to also be the case with GECP. \Cref{fig: compare hist} shows the results of running $10^4$ trials, uniformly sampling $B \sim \B_s(N,\Sigma_S)$ with $N = 2^8$ (a sufficient size chosen for efficient sampling). For each sampled matrix, we computed the GEPP growth factors, along with each GECP growth factors using first the added tolerance parameter (with $\operatorname{tol} = 10^3 \cdot \epsilon_{\operatorname{machine}}$) to ensure the column-major lexicographic tie-breaking strategy was enforced, as well as the GECP growth factors without this added tolerance parameter. In \textit{all} $10^4$ trials, the computed growth factors matched exactly within $1.2790\cdot 10^{-13}$. \Cref{fig: compare hist}(a) shows the overlaid final histograms of the computed growth factors $\rho(B)$ on a logarithmic scaling. For comparison, an alternative growth factor
\begin{equation*}
    \rho_\infty(A) = \frac{\|L\|_\infty \|U\|_\infty}{\|A\|_\infty}
\end{equation*}
was additionally computed for each GE factorization. This was used in particular on the experiments in \cite{PT23}, which presented an analogous full distributional description for $\rho_\infty(B)$ for $B \sim \B_s(N,\Sigma_S)$ when using GEPP. Unlike using the max-norm $\|\cdot \|_{\max}$ (as in $\rho(B)$ and seen in \Cref{fig: compare hist}(a)), the induced $L^\infty$ matrix norm $\|\cdot \|_\infty$ was more sensitive to the corresponding row and column pivots encountered in GECP as seen in the non-overlaying histograms in \Cref{fig: compare hist}(b). Of note, the added randomness of the location of the pivot at each intermediate step in the GECP without tolerance led to a smaller overall computed $L^\infty$-growth factor.

\section{Butterfly Hadamard matrices} \label{sec: had}


As evidenced by \eqref{eq: gf angle vec}, the growth factors for butterfly matrices are maximized when $$\max(|\tan \t_j|,|\cot \t_j|) = |\tan \t_j| = 1$$ for each $j$, in which case $\rho(B) = 2^n = N$. Hence, $\rho(B(\boldsymbol{\t})) \in \B_s(N)$ is maximal whenever $$\boldsymbol{\t} \in \{ \pm \frac\pi2 \pm \frac\pi4\} = \{\frac\pi4, \frac{3\pi}4,\frac{5\pi}4,\frac{7\pi}4\}.$$ In this case, since then $\sqrt 2 \cdot R_{\t} = H \cdot (\pm 1 \oplus \pm 1)$; it follows $H(\boldsymbol{\t}) = \sqrt{N} \cdot B(\boldsymbol{\t})$ is a scaled Hadamard matrix. A similar construction works for any scalar or diagonal butterfly matrix formed only using input angles $\pm \frac\pi2 \pm \frac\pi4$, such that $H(\boldsymbol{\t}) = \sqrt N \cdot B(\boldsymbol{\t})$ is again a Hadamard matrix; this will be formally established in \Cref{prop: butterfly hadamard}. For convention, we will refer to such Hadamard matrices as \textit{butterfly Hadamard matrices}, and we will let $\H_s(N), \H(N),\H^{(D)}_s(N), \H^{(D)}(N)$ then correspond to the butterfly Hadamard matrices formed respectively by requiring $B(\boldsymbol{\t})$ belong to $\B_s(N), \B(N), \B_s^{(D)}(N), \B^{(D)}(N)$. 

The most famous ongoing conjecture with Hadamard matrices corresponds to the existence of such a matrix of all orders divisible by 4. The smallest such order for which no Hadamard matrix has been found (yet) is $768$ \cite{Dok08}. Hadamard matrices have their own rich history in terms of the growth problem. The $n^{O(\log n)}$ GECP growth bound Wilkinson provided was believed to be far from optimal. For nearly 30 years, a popular belief was that the growth was actually bounded by $n$, and that this could be achieved only by Hadamard matrices (see \cite{EdUr23} for further historical background). This was disproved by Edelman in \cite{EdelmanGECP} when he confirmed a $13\times 13$ matrix with growth of about 13.0205, upgrading a previous empirical counterexample using floating-point arithmetic one year earlier by Gould in \cite{Go91} to a true counterexample in exact arithmetic. However, the previous conjectured growth bound of $n$ still remains intact for Hadamard matrices (and orthogonal matrices; see also \cite{p24_orth}).

The Hadamard growth problem further subdivides the target to consider additionally the growth problem on different equivalence classes of Hadamard matrices: two Hadamard matrices $H_1,H_2$ are equivalent if there exists two signed permutation matrices $P_1, P_2$ such that $P_1 H_1 = H_2 P_2$. There is only one Hadamard equivalence class for orders up to 12, while there are 5 distinct classes of order 16, 2 for order 20, 60 for order 24, 487 for 28, and then millions for the next few orders (e.g., 13,710,027 for 32; see \cite{Kharaghani_Tayfeh-Rezaie_2013}).  By considering all Hadamard equivalence classes, the Hadamard growth problem has only been resolved up to $n = 16$ \cite{KrMi09}. The Hadamard growth problem is additionally resolved (using \Cref{prop: torn}) for Sylvester equivalent Hadamard matrices of all orders. A straightforward check shows that $H(\boldsymbol{\t}) \in \H_s(N)$ is equivalent to a Sylvester Hadamard matrix, so that then \Cref{thm: new} provides an additional resolution path for the Sylvester Hadamard growth problem. In this context, as well as with \Cref{prop: lipschitz}, then the simple scalar butterfly growth problem can be viewed as a continuous generalization of the Sylvester Hadamard growth problem.

Another open problem of interest with Hadamard matrices involves the number of such possible matrices of a given size $m$. By considering only the $m!$ permutations of the rows, then a lower bound on the count of Hadamard matrices is $2^{\Omega(m \log m)}$ while an upper bound of $2^{\binom{m+1}2} = 2^{O(m^2)}$ is also easy to attain (see \cite{Fe18} for an overview of this conjecture, along with the current improved upper bound of order $2^{(1-c)m^2/2}$ for a fixed constant $c$; it is further conjectured the lower bound is the correct asymptotic scaling). The butterfly Hadamard matrices produce a rich set of Hadamard matrices that can be used as a simple way to produce random Hadamard matrices when using random input angles from $\pm \frac\pi2 \pm \frac\pi4$ (e.g., consider a uniform angle from this set), as is seen here:
\begin{proposition}
    \label{prop: cnt butterfly had}
    For $N = 2^n$, then $|\H_s(N)| = 2^{n+1} = 2 N$, $|\H(N)| = 2^{3 \cdot 2^{n-1} - 1} = 2^{\frac32 N - 1}$, $|\H_s^{(D)}(N)| = 2^{2^n -n + 1} = 2^{N - n + 1}$, and $|\H^{(D)}(N)| = 2^{2^{n-1} n + 1} = 2^{N n/2 + 1}$. In particular, $|H^{(D)}(N)| = 2^{O(N \log N)}$ matches the conjectured asymptotic bound on the number of Hadamard matrices.
\end{proposition}
\begin{proof}
    If $\alpha_n$ is the respective size for each set of butterfly Hadamard matrices of order $2^n$, then each satisfy $\alpha_1 = 4$ along with the respective recurrence relations \begin{equation*}
    \alpha_{n+1} = \left\{
    \begin{array}{ll} 
    \alpha_1 \cdot \alpha_n/2 = 2 \alpha_n & \mbox{for $\H_s(N)$,}\\ 
    \alpha_1 \cdot \alpha_n^2/2 = 2\alpha_n^2 & \mbox{for $\H(N)$,}\\ 
    \alpha_1^{N/2} \cdot \alpha_n/2 = 2^{2^n-1} \alpha_n & \mbox{for $\H_s^{(D)}(N)$, and}
    \\ 
    \alpha_1^{N/2} \cdot \alpha_n^2/2 = 2^{2^n - 1} \alpha_n^2 & \mbox{for $\H^{(D)}(N)$,}
    \end{array}
    \right.
\end{equation*}
where \eqref{eq:bm_def} can be used to first determine the number of different inputs for each of the left and right factors, which then overcount by 2 as the number of ways that the identity matrix can be factored out (with $-\V I$ terms from both the left and right). A direct check then verifies each of the above counts satisfy each respective recurrence. (Equivalently, one can solve the recurrences using instead $\beta_n = \log_2 \alpha_n$.)
\end{proof}

\begin{remark}
    A different way to equivalently count $|\H_s(N)|$ is to first consider the $4^n$ possible input vectors with $\t_j = \pm \frac\pi2\pm \frac\pi4$. This overcounts by all inputs that have scalar factors that combine to produce the identity matrix. This occurs with an even number of the input Kronecker factors can have a $-\V I_2$ that factors out, which occurs $\sum_{k \operatorname{ even}} \binom{n}k = 2^{n-1}$ times. Hence, the total count is then $4^n/2^{n-1} = 2^{n+1} = 2 N$.
\end{remark}

Popular constructions of Hadamard matrices include the Sylvester and Walsh-Hadamard constructions that work for for $N = 2^n$ or $N = mk$ where a Hadamard matrix is known to exist of order $m$ and $k$. Alternative constructions of Hadamard matrices include the Paley construction, which uses tools from finite fields, and the Williamson construction, which uses the sum of four squares \cite{Pa33,Wi44}. Applications of these only work for very particular orders and do not cover all potential multiple of 4 orders.

The butterfly Hadamard matrices provide an additional method of constructing certain Walsh-Hadamard matrices. As previously mentioned, $\H_s(N)$ yield particular Sylvester equivalent Hadamard matrices. This gives one way to generate a Sylvester Hadamard matrix using the butterfly matrices, with $2N$ possible butterfly Hadamard matrices of order $N$. In light of \Cref{prop: lipschitz}, the particular butterfly Hadamard matrices can also be interpreted as being Hadamard representations for butterfly matrices whose input vectors have $(\cos \t_j,\sin\t_j)$ with particular sign combinations in $\{(\pm 1,\pm 1)\}^M$, with $M = n,N-1,N-1,Nn/2$ for respectively scalar simple, scalar nonsimple, diagonal simple, and diagonal nonsimple butterfly matrices (assuming $|\cos\t_j \sin \t_j| \ne 0$, so that $\t_j$ is not an integer multiple of $\frac\pi2$). Hence, the butterfly Hadamard matrices can be viewed as representing the fixed sector $(\pm 1,\pm 1)^M$ on which the input angle vector maps to.

Rather than fixing given angles, one can generate a Hadamard matrix using (almost) any butterfly matrix combined with the $\sgn$ function, which when applied to a matrix acts componentwise so that $(\sgn(A))_{ij} = \sgn(A_{ij}) = \frac{A_{ij}}{|A_{ij}|}$, where we define $\sgn(0) := 0$. If $A \in (\mathbb R\backslash \{0\})^{n\times m}$ then $\sgn(A) \in \{\pm 1\}^{n\times m}$. In particular, note if $\t_j \not\in\frac\pi2\mathbb Z$ for all $j$, then $\sgn(B(\boldsymbol\t)) \in \{\pm1\}^{N\times N}$ for any scalar or diagonal butterfly matrix.

Note $\sgn(DA) = \sgn(D)\sgn(A)$ if $D$ is a diagonal matrix. It follows then
\begin{equation}
\label{eq: signed simple hadamard}
    \sgn(A\otimes B) = \sgn(A) \otimes \sgn(B), \quad \sgn\left((A \otimes \V I_{N/2})(B \oplus C) \right) = (\sgn(A) \otimes \V I_{N/2})(\sgn(B) \oplus \sgn(C))
\end{equation}
for $A \in \mathbb R^{2\times 2}$ and $B,C \in \mathbb R^{N/2}$. In particular, for $C,S$ diagonal, then
\begin{equation}
\label{eq: signed hadamard}
    \sgn\left(\begin{bmatrix}
    C A_1 & S A_2\\-SA_1 & C A_2
    \end{bmatrix} \right) = \begin{bmatrix}
    \sgn(C) \sgn(A_1) & \sgn(S) \sgn(A_2)\\-\sgn(S)\sgn(A_1) & \sgn(C)\sgn(A_2)
    \end{bmatrix}.
\end{equation}

For $B(\boldsymbol \t) \in \B_s(N)$, from \eqref{eq: signed simple hadamard} we have
\begin{equation}
\sgn(B(\boldsymbol\t)) = \sqrt{N}\cdot B(\hat{\boldsymbol\t}) \quad \mbox{where} \quad \hat \t_j = \frac\pi4\left(2\left\lfloor \frac{2\t_j}{\pi} \right\rfloor+1\right),
\end{equation}
where then $\hat \t$ sends $\t$ to the representative of $\pm \frac\pi2 \pm \frac\pi4$ based on which quadrant $(\cos\t,\sin \t)$ is located. For example, $\hat 1 = \frac\pi4$ and $\hat{4} = \frac{5\pi}4$. Additionally, from \eqref{eq: signed hadamard} we have $\sgn(B)$ is a Walsh-Hadamard matrix for also $B$ a nonsimple scalar butterfly matrix, simple diagonal butterfly matrix and nonsimple diagonal butterfly matrix assuming $\t_j \not\in \frac\pi2\mathbb Z$ for any $j$. The nonsimple scalar butterfly matrix case follows immediately from \eqref{eq: signed hadamard}. The diagonal case is less direct:

\begin{proposition}
\label{prop: butterfly hadamard}
If $B(\boldsymbol \t)$ is a scalar or diagonal order $N$ butterfly matrix with $\t_j \not\in \frac\pi2\mathbb Z$ for all $j$, then $\sgn(B(\boldsymbol\t)) = \sqrt{N} \cdot B(\hat{\boldsymbol\t})$ is a Hadamard matrix.
\end{proposition}
\begin{proof}
The equality in $\sgn(B(\boldsymbol\t)) = \sqrt{N}\cdot B(\hat{\boldsymbol\t})$ follows from \Cref{prop: lipschitz} and the fact $\sgn$ is constant on sectors of $\mathbb R^{N\times N}$. To establish $H = \sgn(B)$ is a Hadamard matrix for $B$ a butterfly matrix, it suffices to consider only  $B = B(\boldsymbol \t) \in \B^{(D)}(N)$ (since each other butterfly ensemble is a subset of this). Since $\t_j \not\in \frac\pi2\mathbb Z$ for all $j$, then $H \in \{\pm 1\}^{N\times N}$. To see $HH^\top = N\V I_N$, we will use induction on $n$. 

First, we will consider $(C,S)(\boldsymbol \t) = \bigoplus_{j=1}^{N/2}(\cos \t_j,\sin\t_j)$. Since $C$ and $S$ are nonsingular diagonal matrices, then $\sgn(C)^2=\sgn(S)^2=\V I_{N/2}$ and $\sgn(C)\sgn(S) = \sgn(S)\sgn(C)$. It follows 
\begin{equation}
\label{prop: base case}
    \begin{bmatrix}
    \sgn(C) & \sgn(S)\\-\sgn(S) & \sgn(C)
    \end{bmatrix}\begin{bmatrix}
    \sgn(C) & \sgn(S)\\-\sgn(S) & \sgn(C)
    \end{bmatrix}^\top = 2\V I_N.
\end{equation}
This establishes the base case for $n=2$. Now assume the result holds for $n-1$. From \eqref{eq: signed hadamard}, we have
\begin{equation}
    H = \begin{bmatrix}
    \sgn(C) & \sgn(S) \\ -\sgn(S) & \sgn(C)
    \end{bmatrix} \begin{bmatrix}
    \sgn(A_1) & \V 0 \\\V 0&\sgn(A_2)
    \end{bmatrix}
\end{equation}
where $A_1,A_2 \in \B^{(D)}(N/2)$. By the inductive hypothesis, $\sgn(A_i)\sgn(A_i)^\top = \frac{N}2\V I_{N/2}$ so that 
\begin{equation}
\label{prop: diag case}
(\sgn(A_1)\oplus \sgn(A_2))(\sgn(A_1)\oplus \sgn(A_2))^\top = \frac{N}2\V I_N.    
\end{equation}
Combining \cref{prop: base case,prop: diag case} yields $HH^\top = N \V I_N$.
\end{proof}

\begin{remark}
    The subsampled randomized Hadamard transformation (SRHT) is popular in data compression, dimensionality reduction, and random algorithms, which is often implemented with a fixed fast Walsh Hadamard matrix and a random diagonal sign matrix (see \cite{Tr11}). This in particular makes uses of the fast $\mathcal O(N \log N)$ matrix-vector multiplication property, while the randomness is presented in the random diagonal sign matrix. Future work can explore potentially added utility of random butterfly Hadamard transformations, that maintain this quick matrix-vector multiplication but with additional randomness injected that may better serve certain structured models.
\end{remark}  

\begin{remark}
    Similar constructions using butterfly matrices initiated with $\O(2)$ instead of $\SO(2)$ generate additional Hadamard matrices, but these are still of order $2^{\mathcal O(N)}$.
\end{remark}

\begin{remark}
    \red{\Cref{prop: butterfly hadamard} yields exact Hadamard matrices only when $4 \mid N$, 
    since Hadamard matrices can only exist in these dimensions. Other butterfly constructions can yield matrices of order $N = m^n$ using generating orthogonal groups $\O(m)$ (see \cite{phd} for general butterfly constructions). For other order butterfly matrices $B$, 
    $\sgn(B)$ produces a $\pm 1$ matrix that cannot be exactly Hadamard if $4 \not\mid N$, 
    but may still serve as an \textit{approximately Hadamard} matrix in the sense of 
    Dong and Rudelson \cite{DongRudelson24}, i.e., a $\pm 1$ matrix with condition 
    number bounded independently of $N$. Whether the butterfly $\sgn$ construction 
    achieves this for other generating $\O(m)$ factors is an interesting direction for future investigation.}
\end{remark}

\section{Conclusions and future directions}

\red{We analyzed GECP growth for random butterfly matrices, identifying conditions under which pivoting strategies align and yield explicit growth factor distributions, and placing classical Hadamard growth phenomena within the butterfly framework. We also introduced butterfly Hadamard matrices via a simple $\sgn$ construction, providing a flexible mechanism for generating large families of Hadamard matrices and recovering the conjectured asymptotic count within the diagonal nonsimple class.}

\red{Several directions remain open. Empirical evidence suggests that the growth factor distribution persists without the monotonicity assumption; proving this would require a refined analysis of GECP pivot dynamics. A related problem is to characterize the permutation factors $P,Q$ for general simple butterfly matrices $B \in \B_s(N)$, particularly their deviations from perfect shuffle structure, with potential connections to permuton theory. Additionally, extending the growth analysis beyond simple butterfly matrices to more general scalar and diagonal butterfly classes  remains an open direction, where current results are primarily empirical.}

\red{The butterfly Hadamard construction also raises algorithmic questions. Replacing the Walsh--Hadamard matrix in SRHT with a random butterfly Hadamard matrix preserves $\mathcal O(N\log N)$ complexity while introducing additional structured randomness; its impact on randomized numerical linear algebra remains to be explored. Extending the growth analysis to $\B(N)$, $\B_s^{(D)}(N)$, $\B^{(D)}(N)$, and to $\O(2)$-based constructions is another natural direction. More generally, replacing $\SO(2)$ with $\O(m)$ yields butterfly matrices of order $N = m^n$. When $N$ is not divisible by $4$, $\sgn(B)$ is not exactly Hadamard, but may produce approximately Hadamard matrices in the sense of \cite{DongRudelson24}, suggesting that the butterfly $\sgn$ construction may extend beyond the power-of-two setting.}

\appendix

\section{Connectedness of butterfly matrices}\label{sec: appendix}

We first recall several elementary properties of the Frobenius norm. It is unitarily invariant and satisfies
\begin{equation}\label{eq: frob tri}
\|U_1V_1-U_2V_2\|_F \le \|U_1-U_2\|_F+\|V_1-V_2\|_F, \quad U_i,V_i\in \U(N).
\end{equation}
Moreover,
\begin{equation}\label{eq: frob add}
\|A\oplus B\|_F^2=\|A\|_F^2+\|B\|_F^2, \qquad
\|A\otimes B\|_F=\|A\|_F\|B\|_F.
\end{equation}

\medskip

We begin with the scalar butterfly model. Using the mixed-product property,
\begin{equation}\label{eq: butterfly svd decomp}
B(\boldsymbol\t)=U_n \Lambda_{\boldsymbol\t} U_n^*,
\end{equation}
where $U_n=\bigotimes^n U$ and $\Lambda_{\boldsymbol\t}=\bigotimes_{j=1}^n \Lambda_{\t_{n-j+1}}$, with
\[
U=\frac{1}{\sqrt{2}}\begin{bmatrix}1&1\\ i&-i\end{bmatrix}, 
\qquad 
\Lambda_\t=\begin{bmatrix}e^{i\t}&\\ &e^{-i\t}\end{bmatrix}.
\]
Since $\Lambda_{\t+\vep}=\Lambda_\t\Lambda_\vep$, we have
\begin{equation}\label{eq: Bs group}
B(\boldsymbol\t+\boldsymbol\vep)=B(\boldsymbol\t)B(\boldsymbol\vep),
\end{equation}
so $\B_s(N)$ is an abelian subgroup of $\SO(N)$.

We also note the bound
\begin{equation}\label{eq: elem bound exp}
1-\prod_{j=1}^n \cos x_j \le \frac12 \|\V x\|_2^2, \qquad \V x\in\mathbb R^n,
\end{equation}
which follows from the elementary bound $1-\cos x \le x^2/2$ applied inductively along with the fact $1-xy=1-x+x(1-y)$.

\begin{lemma}\label{prop: cont simple scalar}
Let $B(\boldsymbol\t)\in\B_s(N)$ and $\boldsymbol\vep\in\mathbb R^n$. Then
\begin{equation}\label{eq: cont simple scalar}
\|B(\boldsymbol\t)-B(\boldsymbol\t+\boldsymbol\vep)\|_F
\le \sqrt{N}\,\|\boldsymbol\vep\|_2.
\end{equation}
\end{lemma}

\begin{proof}
Using \eqref{eq: Bs group} and unitary invariance,
\[
\|B(\boldsymbol\t)-B(\boldsymbol\t+\boldsymbol\vep)\|_F
=\|\V I_N-B(\boldsymbol\vep)\|_F.
\]
Moreover,
\[
B(\boldsymbol\vep)+B(\boldsymbol\vep)^*
=2\prod_{j=1}^n \cos \vep_j\,\V I_N,
\]
so
\begin{align*}
\|\V I_N-B(\boldsymbol\vep)\|_F^2
&=\trace\!\left(2\V I_N-(B(\boldsymbol\vep)+B(\boldsymbol\vep)^*)\right) =2N\Bigl(1-\prod_{j=1}^n \cos \vep_j\Bigr)
\le N\|\boldsymbol\vep\|_2^2,
\end{align*}
by \eqref{eq: elem bound exp}.
\end{proof}

A direct consequence is that for each $j$,
\begin{equation}\label{eq: lemma1}
\|\V I_N-\V I_{N2^{-j}}\otimes B(\vep)\otimes \V I_{2^{j-1}}\|_F^2
\le N\vep^2.
\end{equation}

\medskip

We now pass to block butterfly factors. The models admit factorizations
\[
\B(N)=\prod_{j=1}^n \mathcal D(N,N2^{-j}), \quad
\B_s^{(D)}(N)=\prod_{j=1}^n \mathcal D_s^{(D)}(N,N2^{-j}), \quad
\B^{(D)}(N)=\prod_{j=1}^n \mathcal D^{(D)}(N,N2^{-j}),
\]
where
\begin{align}
\mathcal D(N,k)&=\bigoplus^k \SO(2)\otimes \V I_{N/2k}, \\
\mathcal D_s^{(D)}(N,k)&=\mathcal D_s^{(D)}(N/k,1)\otimes \V I_k
= P_{k,N}\mathcal D(N,N/2k)P_{k,N}^\top, \label{eq: Dk vs Dksd}\\
\mathcal D^{(D)}(N,k)&=\bigoplus^k \mathcal D_s^{(D)}(N/k,1),\label{eq: Dkd vs Dksd}
\end{align}
for $P_{k,N}$ the perfect shuffle permutation matrix such that $A \otimes \V I_k = P_{k,N}(\V I_k \otimes A) P_{k,N}^\top$ for any $A \in \mathbb R^{N/k \times N/k}$.

\begin{lemma}\label{lemma: D_k ineq}
Let $\boldsymbol\vep$ be appropriately dimensioned. Then
\begin{equation}
\|B(\boldsymbol\t)-B(\boldsymbol\t+\boldsymbol\vep)\|_F^2 \le
\begin{cases}
\dfrac{N}{k}\|\boldsymbol\vep\|_2^2, & B\in\mathcal D(N,k), \\[6pt]
2k\|\boldsymbol\vep\|_2^2, & B\in\mathcal D_s^{(D)}(N,k), \\[6pt]
2\|\boldsymbol\vep\|_2^2, & B\in\mathcal D^{(D)}(N,k).
\end{cases}
\label{eq: ineq Dk}
\end{equation}
\end{lemma}

\begin{proof}
Each block model forms an abelian group, so $B(\boldsymbol\t+\boldsymbol\vep)=B(\boldsymbol\t)\B(\boldsymbol\vep)$, and hence
\[
\|B(\boldsymbol\t)-B(\boldsymbol\t+\boldsymbol\vep)\|_F =\|\V I_N-B(\boldsymbol\vep)\|_F.
\]
For $\mathcal D(N,k)$,
\[
\|\V I_N-B(\boldsymbol\vep)\|_F^2
=\sum_{j=1}^k \|\V I_{N/k}-B(\vep_j)\otimes \V I_{N/2k}\|_F^2
\le \frac{N}{k}\|\boldsymbol\vep\|_2^2,
\]
using \eqref{eq: frob add} and \eqref{eq: lemma1}. The remaining bounds follow from \eqref{eq: Dk vs Dksd}--\eqref{eq: Dkd vs Dksd} and unitary invariance.
\end{proof}

\medskip

We now prove the Lipschitz continuity conditions stated in \Cref{prop: lipschitz}.

\begin{proof}[Proof of \Cref{prop: lipschitz}]
The case $\B_s(N)$ is \Cref{prop: cont simple scalar}. For $\B(N)$, write
\[
B(\boldsymbol\t)=\prod_{j=1}^n B_j(\boldsymbol\t_j), \quad B_j(\boldsymbol\t_j)\in\mathcal D(N,N2^{-j}),
\]
and similarly decompose $\boldsymbol\vep$. Then, by \eqref{eq: frob tri}, \Cref{lemma: D_k ineq}, and the Cauchy-Schwarz inequality,
\begin{align*}
\|B(\boldsymbol\t)-B(\boldsymbol\t+\boldsymbol\vep)\|_F
&= \left\|\prod_{j=1}^n B_j(\boldsymbol{\t}_j) - \prod_{j=1}^n B_j(\boldsymbol{\t}_j + \boldsymbol{\vep}_j) \right\|_F\\
&\le \sum_{j=1}^n \|B_j(\boldsymbol\t_j)-B_j(\boldsymbol\t_j+\boldsymbol\vep_j)\|_F\\
&\le \sum_{j=1}^n 2^{j/2}\|\boldsymbol\vep_j\|_2 \\
&\le \Bigl(\sum_{j=1}^n 2^j\Bigr)^{1/2}
\Bigl(\sum_{j=1}^n \|\boldsymbol\vep_j\|_2^2\Bigr)^{1/2}
=\sqrt{2(N-1)}\,\|\boldsymbol\vep\|_2.
\end{align*}
The proofs for $\B_s^{(D)}(N)$ and $\B^{(D)}(N)$ are identical, using the corresponding bounds from \Cref{lemma: D_k ineq}.
\end{proof}

\medskip

For comparison, recall that Givens rotations
\[
G(\t,i,j)=P_{(2\,j)(1\,i)}(B(\t)\oplus \V I_{n-2})P_{(2\,j)(1\,i)}^\top
\]
generate $\SO(n)$. The same argument yields Lipschitz continuity of
\[
\boldsymbol\theta \mapsto A(\boldsymbol\theta)
=\prod_{i<j} G(\t_{\alpha_{ij}},i,j),
\]
with $m=\frac{n(n-1)}2$ parameters.

\begin{proposition}[\cite{p24_orth}]
The map $\boldsymbol\theta\mapsto A(\boldsymbol\theta)\in\SO(N)$ is $\sqrt{N(N-1)}$-Lipschitz.
\end{proposition}

\begin{remark}
Applying \eqref{eq: frob tri} directly to $\B_s(N)$ yields the weaker constant $\sqrt{Nn}$. The sharper $\sqrt{N}$ bound arises from exploiting the Kronecker structure. More generally, this approach produces Lipschitz constants scaling with the square root of the parameter dimension for each butterfly model.
\end{remark}

\bibliographystyle{plain}
\bibliography{references}

\end{document}